\documentclass{article}

\usepackage{amsmath}
\usepackage{amssymb}
\usepackage{amscd}
\usepackage{enumerate}
\usepackage{bbm}
\usepackage{mathrsfs}
\usepackage{amsthm}
\usepackage[colorlinks=true]{hyperref}

\usepackage{todonotes}

\usepackage[utf8]{inputenc}

\newcommand{\gmtw}[1]{\langle #1 \rangle}
\newcommand{\Aone}{{\mathbb{A}^1}}
\newcommand{\Gm}{{\mathbb{G}_m}}
\newcommand{\Gmp}[1]{{\mathbb{G}_m^{\wedge #1}}}
\newcommand{\id}{\operatorname{id}}
\newcommand{\ZZ}{\mathbb{Z}}

\newcommand{\eff}{{\text{eff}}}
\newcommand{\veff}{{\text{veff}}}
\newcommand{\DM}{\mathbf{DM}}

\newcommand{\SH}{\mathbf{SH}}

\newcommand{\Hom}{\operatorname{Hom}}

\newcommand{\hocolim}{\operatorname{hocolim}}

\newcommand{\ul}{\underline}
\newcommand{\heart}{\heartsuit}

\newcommand{\KO}{{KO}}

\newcommand{\iso}{\cong}
\newcommand{\wequi}{\simeq}

\newtheorem{prop}{Proposition}

\newtheorem{lemm}[prop]{Lemma}
\newtheorem{thm}[prop]{Theorem}
\newtheorem*{thm*}{Theorem}

\title{The Generalized Slices of Hermitian K-Theory} 
\author{Tom Bachmann}

\begin{document}

\maketitle

\begin{abstract}
We compute the generalized slices (as defined by Spitzweck-{\O}stv{\ae}r) of
the motivic spectrum $\KO$ (representing hermitian $K$-theory) in terms of
motivic cohomology and (a
version of) generalized motivic cohomology, obtaining good agreement with the
situation in classical topology and the results predicted by
Markett-Schlichting. As an application, we compute the homotopy sheaves of (this
version of) generalized motivic cohomology, which establishes a version of a
conjecture of Morel.
\end{abstract}

\section{Introduction}

$K$-theory was invented by algebraic geometers and taken up by topologists. As a
result of Bott-periodicity, the homotopy groups of the (topological) complex
$K$-theory spectrum $KU$ are alternatingly $\ZZ$ and $0$. Consequently the (sped
up) Postnikov tower yields a filtration of $KU$ with layers all equal to the
Eilenberg-MacLane spectrum $H\ZZ$ (which is also the zeroth Postnikov-layer of the
topological sphere spectrum). From this one obtains the Atiyah-Hirzebruch
spectral sequence, which has the singular cohomology of a space $X$ on the $E_2$
page and converges to the (higher) topological $K$-theory of $X$.

Much research has been put into replicating this picture in algebraic geometry.
In its earliest form, this meant trying to find a cohomology theory for
algebraic varieties called \emph{motivic cohomology} which is related via a
spectral sequence to higher algebraic $K$-theory. There is now a very
satisfactory version of this picture. The motivic analog of the stable homotopy
category $\SH$ is the motivic stable homotopy category $\SH(k)$ \cite[Section
5]{morel2004motivic-pi0}. Following
Voevodsky \cite[Section 2]{voevodsky-slice-filtration}, this category is
filtered by effectivity, yielding a kind
of $\Gm$-Postnikov tower called the \emph{slice filtration} and denoted
\[ \dots \to f_{n+1} E \to f_n E \to f_{n-1} E \to \dots \to E. \]
The cofibres $f_{n+1} E \to f_nE \to s_n E$ are called the \emph{slices} of $E$,
and should be thought of as one kind of replacement of the (stable) homotopy
groups from
classical topology in motivic homotopy theory.

In $\SH(k)$ there are (at least) two special objects (for us): the sphere
spectrum $S \in \SH(k)$ which is the unit of the symmetric monoidal structure,
and the algebraic $K$-theory spectrum $KGL \in \SH(k)$ representing algebraic
$K$-theory. One may show that up to twisting, all the slices of $KGL$ are
isomorphic, and in fact isomorphic to the zero slice of $S$ \cite[Sections 6.4
and 9]{levine2008homotopy}.
Putting $H_\mu \ZZ = s_0 S$, this spectrum can be used to \emph{define} motivic
cohomology, and then the sought-after picture is complete.

Nonetheless there are some indications that the slice filtration is not quite
right in certain situations. We give three examples.
(1) We have said before that the homotopy groups of
$KU$ are alternatingly given by $\ZZ$ and $0$. Thus in order to obtain a
filtration in which all the layers are given by $H\ZZ$, one has to ``speed up''
the Postnikov filtration by slicing ``with respect to $S^2$ instead of $S^1$''.
Since the slice filtration is manifestly obtained by slicing with respect to
$\Gm$ which is (at best) considered an analogue of $S^1$, and yet the layers are
already given at double speed, something seems amiss. (2) In classical topology
there is another version of $K$-theory, namely the $K$-theory of real (not
complex) vector bundles, denoted $KO$. There is also Bott-periodicity, this time
resulting in the computation that the homotopy groups of $KO$ are given by $\ZZ,
\ZZ/2, \ZZ/2, 0, \ZZ, 0, 0, 0$ and then repeating periodically. There is an
analogue of topological $KO$ in algebraic geometry, namely \emph{hermitian
$K$-theory} \cite{hornbostel2005a1} and (also) denoted $\KO \in \SH(k)$. It
satisfies an
appropriate form of Bott periodicity, but this is not captured accurately by its
slices, which are also very different from the topological analog
\cite{rondigs2013slices}. (3) The slice filtration does not always
\emph{converge}. Thus just considering slices is not enough, for example, to
determine if a morphism of spectra is an isomorphism.

Problem (3) has lead Spitzweck-{\O}stv{\ae}r \cite{spitzweck2012motivic} to
define a refined version of the effectivity condition yielding the slice
filtration which they call being ``very effective''. In this article we shall
argue that their filtration also solves issues (1) and (2).

To explain the ideas, recall that the category $\SH(k)^\eff$ is the localising
(so triangulated!)
subcategory generated by objects of the form $\Sigma^\infty X_+$ for $X \in
Sm(k)$ (i.e. no desuspension by $\Gm$). Then one defines $\SH(k)^\eff(n) =
\SH(k) \wedge T^{\wedge n}$ and for $E \in \SH(k)$ the $n$-effective cover $f_n
E \in \SH(k)^\eff(n)$ is the universal object mapping to $E$. (Note that since
$\SH(k)^\eff$ is triangulated, we have $\SH(k)^\eff \wedge T^{\wedge n} = \SH(k)^\eff
\wedge \Gmp{n}$.) In contrast, Spitzweck-{\O}stv{\ae}r define the
subcategory of very effective spectra $\SH(k)^\veff$ to be the subcategory
generated under homotopy colimits and extensions
by $\Sigma^\infty X_+ \wedge S^n$ where $X \in
Sm(k)$ and $n \ge 0$. This subcategory is \emph{not} triangulated! As before we
put $\SH(k)^\veff(n) = \SH(k)^\veff \wedge T^{\wedge n}$. (Note that now,
crucially, $\SH(k)^\veff \wedge T^{\wedge n} \ne \SH(k)^\veff \wedge \Gmp{n}$.)
Then as before the very $n$-effective cover $\tilde{f}_n E \in
\SH(k)^\veff(n)$ is the universal object mapping to $E$. The cofibres
$\tilde{f}_{n+1} E \to \tilde{f}_n E \to \tilde{s}_n E$ are called the
\emph{generalized slices} of $E$.

As pointed out by Spitzweck-{\O}stv{\ae}r, the connectivity of $\tilde{f}_n E$
\emph{in the homotopy $t$-structure} increases with $n$, so the generalized
slice filtration automatically converges. Moreover it is easy to see that
$\tilde{f}_n KGL = f_n KGL$ (i.e. the $n$-effective cover of $KGL$ is
``accidentally'' already very $n$-effective) and thus $\tilde{s}_n KGL =
s_n(KGL)$. This explains how the generalized slice filtration solves problem
(1): we see that the ``$\Gm$-slices'' (i.e. ordinary slices) of $KGL$ agree ``by
accident'' with the ``$T$-slices'' (i.e. generalized slices). But note that $T$
is an analogue of $S^2$, explaining the double-speed convergence.

The main point of this article is that the generalized slices of $\KO$ can be
computed, and have a form which is very similar to the classical analogue, thus
solving problem (2). Of
course this leads to Atiyah-Hirzebruch type spectral sequences for Hermitian
$K$-theory. Heuristically, the generalized slices of $\KO$ are (supposed to be) like the
$S^2$-Postnikov layers of the topological spectrum $\KO$. We thus expect that
they are $4$-periodic (up to twist). Moreover we expect that $\tilde{s}_i$ for
$i \equiv 1, 2, 3 \pmod{4}$ should just ``accidentally'' be ordinary zero-slices
(corresponding to the fact that $\pi_i KO = 0$ for $i = 3, 5, 7$), whereas
$\tilde{s}_0 \KO$ should be an extension of two objects (corresponding to $\pi_1
KO \ne 0 \ne \pi_0 KO$). This is indeed the case:
\begin{thm*}[(see Theorem \ref{thm:main})]
The generalized slices of Hermitian $K$-theory are given as follows:
\begin{equation*}
\tilde{s}_n \KO \wequi T^{\wedge n} \wedge
  \begin{cases}
  \tilde{s}_0(\KO) & n \equiv 0 \pmod{4} \\
  H_\mu \ZZ/2      & n \equiv 1 \pmod{4} \\
  H_\mu \ZZ        & n \equiv 2 \pmod{4} \\
  0                & n \equiv 3 \pmod{4}
  \end{cases}
\end{equation*}
\end{thm*}
What about the
``conglomerate'' $\tilde{s}_0 \KO$? We offer two ways of decomposing it, either
using the effectivity (slice) filtration or using the homotopy $t$-structure. The
relevant triangles are
\[ H_\mu \ZZ/2[1] \to \tilde{s}_0 \KO \to \tilde{H} \ZZ \]
\[ H_W \ZZ \wedge \Gm \to \tilde{s}_0 \KO \to H_\mu \ZZ. \]
See Lemma \ref{lemm:s0-decomp} and Theorem \ref{thm:main} again.
Here $\tilde{H} \ZZ$ is a spectrum which we call \emph{generalized motivic
cohomology}, and $H_W \ZZ$ is a spectrum which we call \emph{Witt-motivic
cohomology}. They can be characterised abstractly as the effective covers of
certain objects in the heart of the homotopy $t$-structure on $\SH(k)$.

The boundary maps in the above two triangles are very interesting and will be
subject of further investigation. Also the computation of generalized slices of
other spectra is an interesting topic which we shall take up in future work.

\paragraph{Relationship to Other Works.}

All our computations are done abstractly in the motivic homotopy category. This
is not really satisfactory, since in general it is essentially impossible to
compute cohomology with coefficients in some abstract spectrum. For the motivic
spectral sequence, there is a parallel and much more computational story to the
one we have outlined above: Voevodsky has defined motivic cohomology via a
category $\DM(k)$ which is reasonably computable \cite{lecture-notes-mot-cohom},
and in fact
motivic cohomology in this sense coincides with motivic cohomology in the sense
of higher Chow groups \cite[Theorem 19.1]{lecture-notes-mot-cohom}
which is certainly very explicit. Grayson
has defined an explicit spectral sequence converging to algebraic
$K$-theory \cite[Section 5]{grayson2005motivic}. Work of Voevodsky
\cite{voevodsky2003zero} and Levine \cite{levine2008homotopy}
shows that the explicit definitions of motivic cohomology mentioned above agree
with the abstract definition $H_\mu \ZZ = s_0 S$. Work of Suslin
\cite{suslin2003grayson}
shows that the Grayson spectral sequence has layers given by the explicit form
of motivic cohomology, which by what we just said is the same as the abstract
form. Work of Garkusha-Panin \cite{garkusha2012motivic} shows that the abstract and explicit
motivic spectral sequences agree.

A similar picture is expected for Hermitian $K$-theory. Calmès-Fasel
\cite{calmes2014finite} have defined a variant $\widetilde{\DM}(k)$ of $\DM(k)$ and
an associated theory $\tilde{H}' \ZZ$ which they call generalized motivic
cohomology. Markett-Schlichting [in preparation]
have defined a version of the
Grayson filtration for Hermitian $K$-theory and they hope to show that the
layers are of the same form as in our Theorem \ref{thm:main}, with
$\tilde{H}\ZZ$ replaced by $\tilde{H}'\ZZ$. The author contends that it will
eventually be shown that $\tilde{H}\ZZ = \tilde{H}'\ZZ$ and that the
Market-Schlichting spectral sequence coincides with the generalized slice
spectral sequence.\footnote{Added later: the isomorphism $\tilde{H}\ZZ \wequi
\tilde{H}'\ZZ$ has now been established and will appear in forthcoming joint work
of the author and Jean Fasel.}

We note that an obvious modification of the Calmés-Fasel construction yields a
spectrum $H_W' \ZZ$. Again the author contends that $H_W' \ZZ = H_W \ZZ$, but this is not
currently known.

\paragraph{More about $\tilde{H}\ZZ$ and $H_W \ZZ$.}

In the mean time, we propose to study the spectra $\tilde{H}\ZZ$ and $H_W \ZZ$
abstractly. Taking intuition from classical topology, i.e. comparing the
two decompositions of $\tilde{s}_0 \KO$ with $(\pi_1 KO, \pi_0 KO) = (\ZZ/2,
\ZZ)$ we see that $\tilde{H}\ZZ$ should be a ``variant'' of $\ZZ$ and $H_W \ZZ$
should be a ``variant'' of $\ZZ/2$. This is a familiar game in motivic homotopy
theory: the standard unoriented variant of $\ZZ$ is the homotopy module
$\ul{K}_*^{MW}$ of Milnor-Witt $K$-theory, i.e. $\ul{\pi}_0 S_*$, and the
standard unoriented variant of $\ZZ/2$ is the homotopy module $\ul{K}_*^W =
\ul{K}_*^{MW}/h$ of Witt $K$-theory \cite[Chapter 3]{A1-alg-top}. (The standard oriented
variants are Milnor $K$-theory $\ul{K}_*^M$ and its mod-2 version.)
Thus the following result
confirms a very optimistic guess:
\begin{thm*}[(Morel's Structure Conjecture; see Theorem \ref{thm:morel-conjecture})]
The homotopy sheaves of $\tilde{H}\ZZ$ and $H_W \ZZ$ are given as follows:
\begin{equation*}
\begin{split}
\ul{\pi}_i(\tilde{H}\ZZ)_* =
\begin{cases}
  \ul{K}_*^{MW} &$i=0$ \\
  \ul{\pi}_i(H_\mu \ZZ)_* &i\ne 0
\end{cases}
\end{split}
\quad\quad
\begin{split}
\ul{\pi}_i(H_W\ZZ)_* =
\begin{cases}
  \ul{K}_*^{W} &$i=0$ \\
  \ul{\pi}_i(H_\mu \ZZ/2)_* &i\ne 0
\end{cases}
\end{split}
\end{equation*}
\end{thm*}

\paragraph{Organisation of this Article.}

In the preliminary Section \ref{sec:recollections} we recall some basic facts
about stable motivic homotopy theory, and in particular the homotopy
$t$-structures.

In Section \ref{sec:effective-spectra} we collect some results about the
category $\SH(k)^\eff$ of effective spectra. In particular we show that it
carries a $t$-structure, show that the effectivization functor $r: \SH(k) \to
\SH(k)^\eff$ is exact, and provide some results about the heart
$\SH(k)^{\eff,\heart}$. Note that by definition $\SH(k)^\veff = \SH(k)^\eff_{\ge
0}$.

In Section \ref{sec:generalized-slice-filtration} we define the generalized
slice filtration and establish some basic results. In particular we show that
there are two canonical ways of decomposing a generalized slice, similar to how
we decomposed $\tilde{s}_0\KO$. We also give precise definitions of the spectra
$H_\mu \ZZ, H_\mu \ZZ/2, \tilde{H}\ZZ$ and $H_W \ZZ$ we use.

In Section \ref{sec:slices-herm-k} we prove our main theorem computing the
generalized slices of the hermitian $K$-theory spectrum $\KO$. This uses
crucially a lemma of Voevodsky \cite[Proposition 4.4]{voevodsky293possible},
the detailed study of the geometry of quaternionic
Grassmannians by Panin-Walter \cite{panin2010quaternionic} and the geometric
representability of symplectic $K$-theory by Quaternionic Grassmannians, as
proved by Panin-Walter \cite{panin2010motivic} and Schlichting-Tripathi
\cite{schlichting2015geometric}.

Finally in Section \ref{sec:morel-conjecture} we compute the homotopy sheaves of
$\tilde{H}\ZZ$ and $H_W\ZZ$. The computation of $\ul{\pi}_i(\tilde{H}\ZZ)_j$ and
$\ul{\pi}_i(H_W \ZZ)_j$ for $i \le 0$ or $j \le 0$
is a rather formal consequence of results in
Section \ref{sec:effective-spectra}. Thus the main work is in computing the higher
homotopy sheaves in positive weights. The basic idea is to play off the two triangles
$H_\mu \ZZ/2[1] \to \tilde{s}_0 \KO \to \tilde{H} \ZZ$ and $H_W \ZZ \wedge \Gm
\to \tilde{s}_0 \KO \to H_\mu \ZZ$ against each other. For example, an immediate
consequence of the first triangle is that $\ul{\pi}_i(\tilde{s}_0 \KO)_0$ is
given by $\ul{GW}$ if $i=0$, by $\ZZ/2$ if $i = 1$, and by $0$ else. This
implies that $\ul{\pi}_1(H_W \ZZ)_1 = \ZZ/2$, which is a very special case of
Theorem \ref{thm:morel-conjecture}. The general case proceeds along the same
lines. We should mention that this pulls in many more dependencies than the
previous sections, including the resolution of the Milnor conjectures and the
computation of the motivic Steenrod algebra.

\paragraph{Acknowledgements.}
The main impetus for writing this article was a talk by Marco Schlichting where
he presented his results about a Grayson-type spectral sequence for Hermitian
$K$-theory, and Markus Spitzweck's question if this spectral sequence can be
recovered using the generalized slice filtration.

I would also like to thank Fabien Morel for teaching me most of the things I know
about motivic homotopy theory, and for encouragement and support during my
attempts at proving these results. Finally I thank Benjamin Antieau for comments
on a draft of this article.

\paragraph{Conventions.}
Throughout, $k$ is perfect base field. This is because
we will make heavy use of the
homotopy $t$-structure on $\SH(k)$, the heart of which is the category of
homotopy modules \cite[Section 5.2]{morel-trieste}. We denote unit of the
symmetric monoidal structure on $\SH(k)$ by $S$, this is also known as the motivic
sphere spectrum.

We denote by $Sm(k)$ the
category of smooth $k$-schemes. If $X \in Sm(k)$ we write $X_+ \in
Sm(k)_*$ for the pointed smooth scheme obtained by adding a disjoint base point.

We use \emph{homological} grading for our $t$-structures, see for example
\cite[Definition 1.2.1.1]{lurie-ha}.

Whenever we say ``triangle'', we actually mean ``distinguished triangle''.

\section{Recollections on Motivic Stable Homotopy Theory}
\label{sec:recollections}

We write $\SH^{S^1}(k)$ for the $S^1$-stable motivic homotopy
category \cite[Section 4.1]{morel-trieste} and $\SH(k)$ for the
$\mathbb{P}^1$-stable motivic homotopy category \cite[Section
5.1]{morel-trieste}. We let $\Sigma^\infty_{S^1}: Sm(k)_* \to \SH^{S^1}(k)$ and
$\Sigma^\infty: Sm(k)_* \to \SH(k)$ denote the infinite suspension spectrum
functors. Note that there exists an essentially unique adjunction
\[ \Sigma^\infty_s: \SH^{S^1}(k) \leftrightarrows \SH(k): \Omega^\infty_s \]
such that $\Sigma^\infty_s \circ \Sigma^\infty_{S^1} \iso \Sigma^\infty$.

For $E \in \SH^{S^1}(k)$ and $i \in \ZZ$ we define
$\ul{\pi}_i(E)$ to be the Nisnevich sheaf on $Sm(k)$ associated with the
presheaf $X \mapsto [\Sigma^\infty_{S^1} X_+ \wedge S^i, E]$. For $E \in \SH(k)$
and $i, j \in \ZZ$ we put $\ul{\pi}_i(E)_j = \ul{\pi}_i(\Omega^\infty_s(E \wedge
\Gmp{j})$. We define
\begin{align*}
\SH^{S^1}(k)_{\ge 0} &= \{E \in \SH^{S^1}(k) | \ul{\pi}_i(E) = 0 \text{ for all } i < 0\} \\
\SH^{S^1}(k)_{\le 0} &= \{E \in \SH^{S^1}(k) | \ul{\pi}_i(E) = 0 \text{ for all } i > 0\} \\
\SH(k)_{\ge 0} &= \{E \in \SH(k) | \ul{\pi}_i(E)_j = 0 \text{ for all } i < 0, j \in \ZZ\} \\
\SH(k)_{\le 0} &= \{E \in \SH(k) | \ul{\pi}_i(E)_j = 0 \text{ for all } i > 0, j \in \ZZ\}.
\end{align*}
As was known already to Voevodsky, this defines
$t$-structures on $\SH^{S^1}(k), \SH(k)$ \cite[Theorems 4.3.4 and 5.2.6]{morel-trieste}, called the
\emph{homotopy $t$-structures}. The most important ingredient in the proof of
this fact is the \emph{stable connectivity theorem}. The unstable proof in
\cite[Lemma 3.3.9]{morel-trieste} is incorrect; this has been fixed in
\cite[Theorem 6.1.8]{morel2005stable}. It implies that if $X \in Sm(k)$ then
$\Sigma^\infty X_+ \in \SH(k)_{\ge 0}$ and $\Sigma^\infty_{S^1} X_+ \in
\SH^{S^1}(k)_{\ge 0}$ \cite[Examples 4.1.16 and 5.2.1]{morel-trieste}.
If $E \in \SH(k)$ then we denote its truncations by $E_{\ge 0} \in \SH(k)_{\ge
0}, E_{\le 0} \in \SH(k)_{\le 0}$ and so on. We will not explicitly use the
truncation functors of $\SH^{S^1}(k)$, and so do not introduce a notation.

The hearts $\SH^{S^1}(k)^\heart$,
$\SH(k)^\heart$ can be described explicitly. Indeed $\SH^{S^1}(k)^\heart$ is
equivalent to the category of Nisnevich sheaves of
abelian groups which are strictly homotopy invariant (i.e. sheaves $F$ such that
the map $H^p_{Nis}(X, F) \to
H^p_{Nis}(X \times \Aone, F)$ obtained by pullback along the projection $X
\times \Aone \to X$ is an isomorphism, for every $X \in Sm(k)$) \cite[Lemma 4.3.7(2)]{morel-trieste}.
On the other hand $\SH(k)^\heart$ is equivalent to
the category of \emph{homotopy modules} \cite[Theorem 5.2.6]{morel-trieste}. Let us
recall that a homotopy module $F_*$ consists of a sequence of sheaves $F_i \in
Shv(Sm(k)_{Nis})$ which are strictly homotopy invariant, and isomorphisms $F_i
\to (F_{i+1})_{-1}$. Here for a sheaf $F$ the \emph{contraction} $F_{-1}$ is as
usual defined as $F_{-1}(X) = F(X \times (\Aone \setminus 0)) / F(X)$. The
morphisms of homotopy modules are the evident compatible systems of morphisms.
One then shows that in fact for $E \in \SH(k)$, the homotopy sheaves
$\ul{\pi}_i(E)_*$ form (for each $i$) a homotopy module in a natural way
\cite[Lemma 5.2.5]{morel-trieste}.

We will mostly not distinguish the category $\SH(k)^\heart$ from the
(equivalent) category of homotopy modules, and so may write things like ``let
$F_* \in \SH(k)^\heart$ be a homotopy module''.

Because there can be some confusion about the meaning of epimorphism and so on
when several abelian categories are being used at once, let us include the
following observation. It implies in particular that not much harm will come from
confusing for $E \in \SH(k)$ the homotopy module $\ul{\pi}_i(E)_* \in
\SH(k)^\heart$ with the family of Nisnevich sheaves $(i \mapsto
\ul{\pi}_i(E)_i)$.

\begin{lemm} \label{lemm:exactness}
Write $Ab(Shv(Sm(k)_{Nis}))$ for the category of Nisnevich sheaves of abelian groups
on $Sm(k)$, and $Ab(Shv(Sm(k)_{Nis}))^\ZZ$ for the category of $\ZZ$-graded families
of sheaves of abelian groups.
\begin{enumerate}[(1)]
\item The category $\SH^{S^1}(k)^\heart$ has all limits and colimits and the
      functor $\SH^{S^1}(k)^\heart \to Ab(Shv(Sm(k)_{Nis})), E \mapsto \ul{\pi}_0(E)$ is fully faithful
      and preserves limits and colimits.
\item The category $\SH(k)^\heart$ has all limits and colimits and the
      functor $\SH(k)^\heart \to Ab(Shv(Sm(k)_{Nis}))^\ZZ, E \mapsto
      (\ZZ \ni i \mapsto \ul{\pi}_0(E)_i)$ is conservative and preserves limits and colimits.
\end{enumerate}
\end{lemm}
In particular both functors are exact and detect epimorphisms.
Let us also note that a conservative exact functor is faithful
(two morphisms are equal if and only if their equaliser maps isomorphically to
the source).

Before the proof we have two lemmas, which surely must be well-known.

\begin{lemm} \label{lemm:coproducts-heart}
Let $\mathcal{C}$ be a $t$-category and write $j: \mathcal{C}^\heart \to
\mathcal{C}$ for the inclusion of the heart. Let $\{E_i \in \mathcal{C}\}_{i \in
I}$ be a family of objects. If $\bigoplus_i j(E_i) \in \mathcal{C}$ exists then
$\bigoplus_i E_i \in \mathcal{C}^\heart$  exists and is given by $(\bigoplus_i
j(E_i))_{\le 0}$. Similarly, if $\prod_i j(E_i) \in \mathcal{C}$ exists then
$\prod_i E_i \in \mathcal{C}^\heart$  exists and is given by $(\prod_i
j(E_i))_{\ge 0}$.
\end{lemm}
\begin{proof}
The second statement is dual to the first (under passing to opposite
categories), so we need only prove the latter. Note first that $\bigoplus_i
j(E_i) \in \mathcal{C}_{\ge 0}$. Indeed if $E \in \mathcal{C}_{< 0}$ then
$[\bigoplus_i j(E_i), E] = \prod_i[j(E_i), E] = 0$. Consequently if $E \in
\mathcal{C}^\heart$ then $[(\bigoplus_i j(E_i))_{\le 0}, E] = [\bigoplus_i
j(E_i), jE] = \prod_i [E_i, E]$, since $jE \in \mathcal{C}_{\le 0}$ and $j$ is
fully faithful. This concludes the proof.
\end{proof}

Let $\mathcal{C}, \mathcal{D}$ be provided with subcategories $\mathcal{C}_{\ge
0}, \mathcal{C}_{\le 0}, \mathcal{D}_{\ge 0}, \mathcal{D}_{\le 0}$; for example
$\mathcal{C}, \mathcal{D}$ could be $t$-categories.
A functor $F: \mathcal{C} \to \mathcal{D}$ is called \emph{right} (respectively
\emph{left}) \emph{t-exact} if $F(\mathcal{C}_{\ge 0}) \subset \mathcal{D}_{\ge
0}$ (respectively $F(\mathcal{C}_{\le 0}) \subset \mathcal{D}_{\le 0}$). It is
called t-exact if it is both left and right t-exact.

\begin{lemm} \label{lemm:pres-lim-colim}
Let $F: \mathcal{C} \leftrightarrows \mathcal{D}: U$ be an adjunction of
$t$-categories, and assume that $U$ is $t$-exact. Then the induced functor
$U^\heart: \mathcal{D}^\heart \to \mathcal{C}^\heart$ preserves limits and finite colimits.
If $\mathcal{C}$ is compactly generated, $F$ preserves compact objects, and
$\mathcal{D}$ has arbitrary coproducts, then $U^\heart$ preserves all colimits.
\end{lemm}
\begin{proof}
There is an induced adjunction $F^\heart: \mathcal{C}^\heart \leftrightarrows
\mathcal{D}^\heart: U^\heart$, cf. \cite[Proposition 1.3.17
(iii)]{beilinson1982faisceaux}. It follows that $U^\heart$ preserves all limits.
Since $U$ is right $t$-exact, $U^\heart$ is right exact, i.e. preserves finite
colimits \cite[Proposition 1.3.17 (i)]{beilinson1982faisceaux}.

Under the additional assumptions, $U$ preserves arbitrary coproducts, and so
Lemma \ref{lemm:coproducts-heart} implies that $\mathcal{D}^\heart$ has
arbitrary coproducts and that $U^\heart$ preserves them.
The result follows since all colimits can be built from coproducts and finite
colimits.
\end{proof}

\begin{proof}[of Lemma \ref{lemm:exactness}]
The category $\SH(k)$ is compactly generated \cite[Proposition
6.4(3)]{hoyois-equivariant}, and hence has all products and coproducts. It
follows from Lemma \ref{lemm:coproducts-heart} that $\SH(k)^\heart$ has all
products and coproducts, and hence all limits and colimits. The same argument
applies to $\SH^{S^1}(k)^\heart$.

Let $\SH^{S^1}_s(k)$ denote the stable Nisnevich-local homotopy category, built
in the same way as $\SH^{S^1}(k)$, but without performing $\Aone$-localization. It is
also compactly generated. Then
for $E \in \SH^{S^1}_s(k)$ we may define $\ul{\pi}_i(E)$ in just the same way as
before, and we may also define $\SH^{S^1}_s(k)_{\ge 0}, \SH^{S^1}_s(k)_{\le 0}$ in
the same way as before. Then $\SH^{S^1}_s(k)$ is a $t$-category with heart
$Shv(Sm(k)_{Nis})$ \cite[Proposition 3.3.2]{morel2005stable}.
We have the localization adjunction $L: \SH^{S^1}_s(k)
\leftrightarrows \SH^{S^1}(k): i$. By construction $i$ is fully faithful and
$t$-exact. The functor from (1) is given by $i^\heart$, so in particular is
fully faithful. It preserves all limits and colimits by Lemma
\ref{lemm:pres-lim-colim}.

To prove (2), denote the functor by $u$.
Consider for $d \in \ZZ$ the adjunction $\Sigma_s^{\infty+d}: \SH^{S^1}(k)
\leftrightarrows \SH(k): \Omega^{\infty+d}_s$ given by $\Sigma^{\infty+d}_s(E) =
\Sigma^\infty_s(E) \wedge \Gmp{d}$, $\Omega^{\infty+d}_s(F) =
\Omega^{\infty}_s(F \wedge \Gmp{-d})$. Then $\Omega^{\infty + d}$ is
$t$-exact by construction, and so Lemma \ref{lemm:pres-lim-colim} applies. Note
that for $E \in \SH(k)^\heart$ we have $u(E)_d = i^\heart
(\Omega^{\infty-d}_s)^\heart(E)$, where $i^\heart$ is the functor from (1). It
follows that $E \mapsto u(E)_d$ preserves all limits and colimits, and hence so
does $u$. Note also that $u$ detects
zero objects \cite[Proposition 5.1.14]{morel-trieste}, and hence is conservative
(since it detects vanishing of kernel and cokernel of a morphism).
\end{proof}

\section{The Category of Effective Spectra}
\label{sec:effective-spectra}

We write $\SH(k)^\eff$ for the localising subcategory of $\SH(k)$ generated by
the objects
$\Sigma^\infty X_+,$ with $X \in Sm(k)$. By Neeman's version of Brown
Representability, the inclusion $i: \SH(k)^\eff \to \SH(k)$ has a right adjoint
which we denote by $r$.

For $E \in \SH(k)^\eff$ we let $\ul{\pi}_i(E)_0 \in Shv(Sm(k)_{Nis})$ denote
$\ul{\pi}_i(E)_0 := \ul{\pi}_i(iE)_0$. In general we may drop application of the
functor $i$ when no confusion seems likely. We define
\begin{align*}
\SH(k)^\eff_{\ge 0} &= \{E \in \SH(k)^\eff | \ul{\pi}_i(E)_0 = 0 \text{ for all } i < 0\} \\
\SH(k)^\eff_{\le 0} &= \{E \in \SH(k)^\eff | \ul{\pi}_i(E)_0 = 0 \text{ for all } i > 0\}.
\end{align*}

Some or all of the following was already known to Spitzweck-{\O}stv{\ae}r
\cite[paragraph before Lemma 5.6]{spitzweck2012motivic}.

\begin{prop} \label{prop:effective-t-structure}
\begin{enumerate}[(1)]
\item The functors $\ul{\pi}_i(\bullet)_0: \SH(k)^\eff \to Shv(Sm(k)_{Nis})$ form a
  conservative collection.
\item The category
  $\SH(k)^{\eff}_{\ge 0}$ is generated under homotopy colimits and extensions
  by $\Sigma^\infty X_+ \wedge S^n$, where $n \ge 0, X \in Sm(k)$.
\item The functor $r$ is $t$-exact and $i$ is
  right-$t$-exact.
\item The subcategories $\SH(k)^\eff_{\ge 0}, \SH(k)^\eff_{\le 0}$
  constitute a non-degenerate $t$-structure on $\SH(k)^\eff$.
\end{enumerate}
\end{prop}
\begin{proof}
For $X \in Sm(k)$ and $E \in \SH(k)$ we have the strongly convergent Nisnevich descent
spectral sequence $H^p_{Nis}(X, \ul{\pi}_{-q}(E)_0) \Rightarrow [\Sigma^\infty X_+,
E[p+q]]$. Consequently if $E \in \SH(k)^\eff$ and $\ul{\pi}_i(E)_0 = 0$ for all
$i$ then $[\Sigma^\infty X_+, E[n]] = 0$ for all $X \in Sm(k)$ and all $n$.
It follows that $E=0$, since the $\Sigma^\infty X_+$ generate $\SH(k)^\eff$ as
a localizing subcategory (by
definition). Thus the $\ul{\pi}_i(\bullet)_0$ form a conservative collection, i.e.
we have proved (1).

As recalled in the previous section, we have $\Sigma^\infty
X_+ \wedge S^n \in \SH(k)^\eff_{\ge n}$ for $n \ge 0$. Thus if $E \in
\SH(k)^{\eff}_{\ge 0}$, then the homotopy sheaves $\ul{\pi}_i(E)_0$ can be killed off by
attaching cells of the form $\Sigma^\infty X_+ \wedge S^n$ for $n \ge 0$, $X \in
Sm(k)$.
Consequently $\SH(k)^{\eff}_{\ge 0}$ is generated under homotopy colimits and
extensions by objects of the form claimed in (2). We give more details on this
standard argument at the end of the proof.

It follows from adjunction
that for $E \in \SH(k)$ we have $\ul{\pi}_i(r(E))_0 =
\ul{\pi}_i(E)_0$. Consequently $r$ is $t$-exact.
Since $i$ is a left adjoint it commutes with homotopy colimits and so
$i(\SH(k)^\eff_{\ge 0}) \subset \SH(k)_{\ge 0}$ by (2), i.e. $i$ is
right-$t$-exact. Thus we have shown (3).

It remains to show (4), i.e. that we have a non-degenerate $t$-structure. If $E \in
\SH(k)^\eff$ then $riE \wequi E$. Since $r$ is $t$-exact, the
 triangle $r[(iE)_{\ge 0}] \to riE \wequi E \to
r[(iE)_{< 0}]$ coming from the decomposition of $iE$ in the homotopy
$t$-structure is a decomposition of $E$ into non-negative and negative part as
required for a $t$-structure.

Next we need to show that if $E \in \SH(k)^\eff_{>0}$
and $F \in \SH(k)^\eff_{\le 0}$ then $[E, F] = 0$. The natural map $F \to
r[(iF)_{\le 0}]$ induces an isomorphism on all $\ul{\pi}_i(\bullet)_0$,
so is a weak equivalence (by the conservativity result (1)). Thus
$[E, F] = [E, r[(iF)_{\le 0}]] = [iE, (iF)_{\le 0}] = 0$ since $i$ is right-$t$-exact
and so $iE \in \SH(k)_{>0}$.

We have thus shown that $\SH(k)^\eff_{\ge 0}, \SH(k)^\eff_{< 0}$ form a
$t$-structure. It is non-degenerate by (1). This concludes the proof.

\paragraph{Details on killing cells.} We explain in more detail how to prove
(2). Let $\mathcal{C}$ be the
subcategory of $\SH(k)^\eff$ generated under homotopy colimits and extensions by
$\Sigma^\infty X_+ \wedge S^n$, where $n \ge 0, X \in Sm(k)$. We wish to show
that $\SH(k)^\eff_{\ge 0} \subset \mathcal{C}$. As a first step, I claim that if $E
\in \SH(k)^\eff_{\ge n}$ (with $n \ge 0$) there exists $R(E) \in \mathcal{C}
\cap \SH(k)^\eff_{\ge n}$
together with $R(E) \to E$ inducing a surjection on $\ul{\pi}_i(\bullet)_0$ for all $i \ge
0$. Indeed, just let $R(E)$ be the sum $\bigoplus_{\Sigma^\infty X_+ \wedge S^k
\to E} \Sigma^\infty X_+ \wedge S^k$, where the sum is over $k \ge n$, a
suitably large set of varieties $X$, and all maps in $\SH(k)$ as indicated.

Now let $E \in \SH(k)^\eff_{\ge 0}$. We shall construct a diagram $E_0 \to E_1
\to \dots \to E$ with $E_i \in \mathcal{C}$ and $E_i \to E$ inducing an
isomorphism on $\ul{\pi}_j(\bullet)_0$ for all $j < i$. Clearly then $\hocolim_i E_i \to E$
is an equivalence, showing that $E \in \mathcal{C}$, and concluding the proof.

We shall also arrange that $\ul{\pi}_j(E_i)_0 \to
\ul{\pi}_j(E)_0$ is surjective for all $j$ and $i$. Take $E_0 =
R(E)$. Suppose that $E_i$ has been constructed and let us construct $E_{i+1}$.
Consider the homotopy fibre $F \to E_i \to E$. Then $F \in \SH(k)^\eff_{\ge i}$
and $\ul{\pi}_i(F)_0 \to \ul{\pi}_i(E_i)_0 \to \ul{\pi}_i(E)_0 \to 0$ is an exact
sequence (*). Let $E_{i+1}$ be a cone on the composite $R(F) \to F \to E_i$. Since
the composite $R(F) \to F \to E_i \to E$ is zero, the map $E_i \to E$ factors
through $E_i \to E_{i+1}$. It is now easy to see, using (*), that $E_{i+1}$ has
the desired properties.
\end{proof}

\paragraph{Remark.} The paragraph on killing cells in fact shows that
$\SH(k)^\veff$ is generated by $\Sigma^\infty_+ Sm(k)$ under homotopy colimits; no
extensions are needed. We will not use this observation.

\paragraph{Terminology.} In order to distinguish the $t$-structure on
$\SH(k)^\eff$ from the $t$-structure of $\SH(k)$, we will sometimes call the
former the \emph{effective (homotopy) $t$-structure}. We denote the truncations
of $E \in \SH(k)^\eff$ by $E_{\ge_e 0} \in \SH(k)^\eff_{\ge 0}, E_{\le_e 0} \in
\SH(k)^\eff_{\le 0}$ and so on. For $E \in \SH(k)^\eff$, we denote by
$\ul{\pi}_i^\eff(E) \in \SH(k)^{\eff,\heart}$ the homotopy objects. We prove
below that the functor $\ul{\pi}_0(\bullet)_0: \SH(k)^{\eff,\heart} \to Ab(Shv(Sm(k)_{Nis}))$ is conservative
and preserves all limits and colimits. It is thus usually no problem to confuse
$\ul{\pi}_i^\eff(E)$ and $\ul{\pi}_i(E)_0$.

\paragraph{Remark.} By Proposition \ref{prop:effective-t-structure}(3), we have
$\SH(k)^\eff_{\ge 0} \subset \SH(k)_{\ge 0} \cap \SH(k)^\eff$. Since the reverse
inclusion is clear by definition, we conclude that $\SH(k)^\eff_{\ge 0} =
\SH(k)_{\ge 0} \cap \SH(k)^\eff$.

\paragraph{Remark.}
We call the heart $\SH(k)^{\eff,\heart}$ the category of
\emph{effective homotopy modules}. We show below that $i^\heart:
\SH(k)^{\eff,\heart} \to \SH(k)^\heart$ is fully faithful, justifying this
terminology. Note, however, that this is \emph{not} the same category as $\SH(k)^\heart
\cap \SH(k)^\eff$.
It follows from work of Garkusha-Panin
\cite{garkusha2015homotopy} that $\SH(k)^{\eff,\heart}$ is equivalent to the category of homotopy
invariant, quasi-stable,
Nisnevich sheaves of abelian groups with linear framed transfers. We contend
that this category is equivalent to the category of homotopy invariant Nisnevich
sheaves with generalized transfers in the sense of Calmès-Fasel
\cite{calmes2014finite} and also in the sense of Morel
\cite[Definition 5.7]{morel-friedlander-milnor}.

\paragraph{} Except for Proposition \ref{prop:more-iheart}(3) below,
the remainder of this section is not used in the computation of the
generalized slices of hermitian K-theory, only in the last section.

\begin{prop} \label{prop:more-iheart}
\begin{enumerate}[(1)]
\item For $E \in
  \SH(k)^\eff_{\ge 0}$ we have $\ul{\pi}_0(iE)_* = i^\heart \ul{\pi}_0^\eff(E)$,
  where $i^\heart: \SH(k)^{\eff,\heart} \to \SH(k)^\heart$ is the induced
  functor $i^\heart(M) = (iM)_{\le 0}$.
\item The functor $i^\heart: \SH(k)^{\eff,\heart} \to \SH(k)^\heart$ is fully
  faithful.
\item The category $\SH(k)^{\eff,\heart}$ has all limits and colimits, and 
   functor $\SH(k)^{\eff,\heart} \to Ab(Shv(Sm(k)_{Nis})), E \mapsto
  \ul{\pi}_i(E)_0$ is conservative and preserves limits and colimits.
\end{enumerate}
\end{prop}
\begin{proof}
If $E \in \SH(k)^\eff_{\ge 0}$ then we
have the triangle $E_{\ge_e 1} \to E \to \ul{\pi}_0^\eff E$ and consequently get the
triangle $i(E_{\ge_e
1}) \to iE \to i\ul{\pi}_0^\eff E$. But $i$ is right-$t$-exact, so the end of the
associated long exact sequence of homotopy shaves is $0 = \ul{\pi}_0(iE_{\ge_e
1})_* \to \ul{\pi}_0(iE)_* \to \ul{\pi}_0(i\ul{\pi}_0^\eff(E))_* = i^\heart
\ul{\pi}_0^\eff E \to \ul{\pi}_{-1}(i(E_{\ge_e 1}))_* = 0$, whence the claimed
isomorphism of (1).

Let us now prove (2). If $E \in \SH(k)^{\eff, \heart}$ then $E = riE \iso
(riE)_{\le_e 0} \iso r[(iE)_{\le 0}] = ri^\heart(E)$, where the last equality holds by
definition, and the second to last one by $t$-exactness of $r$.
Thus $ri^\heart \iso \id$. Consequently if $E, F \in \SH(k)^{\eff,\heart}$
then $[i^\heart E, i^\heart F] = [E, ri^\heart F] \iso [E, F]$, so $i^\heart$ is
fully faithful as claimed. Here we have used the well-known fact that a $t$-exact
adjunction between triangulated categories induces an adjunction of the hearts
\cite[Proposition 1.3.17 (iii)]{beilinson1982faisceaux}.

Now we prove (3). Since $\SH(k)^\eff$ is compactly generated, existence of
limits and colimits in $\SH(k)^{\eff,\heart}$
follows from Lemma \ref{lemm:coproducts-heart}. Consider the
adjunction $\Sigma^\infty_s: \SH^{S^1}(k) \leftrightarrows \SH(k)^\eff:
\Omega^\infty_s$. Then Lemma \ref{lemm:pres-lim-colim} applies and we find that
$\SH(k)^{\eff,\heart} \to \SH^{S^1}(k)^\heart$ preserves limits and colimits.
Since $\SH^{S^1}(k)^\heart \to Ab(Shv(Sm(k)_{Nis}))$ preserves limits and colimits
by Lemma \ref{lemm:exactness}, we conclude that $\SH(k)^{\eff,\heart} \to
Ab(Shv(Sm(k)_{Nis}))$ also preserves limits and colimits.
Since the functor also detects zero
objects by Proposition \ref{prop:effective-t-structure}(1), we conclude that
it is conservative.
\end{proof}

\paragraph{Remark.}
Parts (1) and (2) of the above proof do not use any special properties of $\SH(k)$ and in fact
show more generally the following: If $\mathcal{C}, \mathcal{D}$ are
presentable stable $\infty$-categories provided with $t$-structures and
$\mathcal{C} \to \mathcal{D}$ is a right-$t$-exact, fully faithful functor, then
the induced functor $\mathcal{C}^\heart \to \mathcal{D}^\heart$ is fully
faithful (in fact a colocalization). This was pointed out to the author by Benjamin
Antieau.

\paragraph{}
Thus the functor $i^\heart$ embeds $\SH(k)^{\eff,\heart}$ into $\SH(k)^\heart$,
explaining our choice of the name ``effective homotopy module''.
We will call a homotopy module $F_* \in \SH(k)^\heart$ \emph{effective} if it is
in the essential image of $i^\heart$, i.e. if there exists $E \in
\SH(k)^{\eff,\heart}$ such that $i^\heart E \iso F$.

If $F_*$ is a homotopy module then we denote by $F_*\gmtw{i}$ the homotopy module
$\ul{\pi}_0(F \wedge \Gmp{i})_*$, which satisfies $F\gmtw{i}_* = F_{*+i}$
and has the same structure maps (just shifted by $i$ places).

\begin{lemm} \label{lemm:htpy-mods-eff}
The homotopy module $\ul{K}_*^{MW}$ of Milnor-Witt K-theory is effective.
Moreover if $F_*$ is an effective homotopy module then so are $F_*\gmtw{i}$ for all
$i\ge 0$. Also cokernels of morphisms of effective homotopy modules are
effective, as are (more generally) colimits of effective homotopy modules.
\end{lemm}
In particular the following homotopy modules are effective: $\ul{K}_*^M =
coker(\eta: \ul{K}_*^{MW}\gmtw{1} \to \ul{K}_*^{MW})$,
$\ul{K}_*^W = coker(h: \ul{K}^{MW}_* \to \ul{K}^{MW}_*)$ as are e.g. $K_*^M/p,
K_*^{MW}[1/p]$ etc.
\begin{proof}
The sphere spectrum $S$ is effective, non-negative and satisfies
$\ul{\pi}_0(S)_* = \ul{K}_*^{MW}$. Hence Milnor-Witt K-theory is effective by
Proposition \ref{prop:more-iheart} part (1).

Now let $E$ in $\SH(k)^{\eff, \heart}$. Then for $i \ge 0$ we have $i(E \wedge
\Gmp{i}) \in \SH(k)_{\ge 0}$ and so $E \wedge \Gmp{i} \in \SH(k)^\eff_{\ge 0}$.
Consequently $(i^\heart E)\gmtw{i} = E_{\le 0} \wedge \Gmp{i} = (E \wedge \Gmp{i})_{\le 0} = i^\heart[(E \wedge
\Gmp{i})_{\le_e 0}]$ is effective. Here the last equality is by
Proposition \ref{prop:more-iheart}(1).

Let $E \to F \in \SH(k)^{\eff,\heart}$ be a morphism and form the right exact
sequence $E \to F \to C \to 0$. Since $i^\heart$ has a right adjoint it is right
exact, whence $i^\heart E \to i^\heart F \to i^\heart C \to 0$ is right exact.
It follows that the cokernel of $i^\heart(E) \to i^\heart(F)$ is $i^\heart(C)$,
which is effective.

Finally $i^\heart$ preserves colimits, again since it has a right adjoint, so
colimits of effective homotopy modules are effective by a similar argument.
\end{proof}

\section{The Generalized Slice Filtration}
\label{sec:generalized-slice-filtration}

We put $\SH(k)^\veff = \SH(k)^\eff_{\ge 0}$ and for $n \in \ZZ$, $\SH(k)^\eff(n)
= \SH(k)^\eff \wedge T^{\wedge n}$, $\SH(k)^\veff(n) = \SH(k)^\veff \wedge
T^{\wedge n}$. Here $T = \Aone/\Gm \wequi (\mathbb{P}^1, \infty) \wequi S^1
\wedge \Gm$ denotes the Tate object. These are the categories of (very)
$n$-effective spectra (we just say ``(very) effective'' if $n=0$).

Write $i_n: \SH(k)^\eff(n) \to \SH(k)$ for the inclusion, $r_n: \SH(k) \to
\SH(k)^\eff(n)$ for the right adjoint, put $f_n = i_n r_n$ and define $s_n$ as
the cofibre $f_{n+1} E \to f_n E \to s_n E$. This is of course the slice
filtration \cite[Section 2]{voevodsky-slice-filtration}.

Similarly we write $\tilde{i}_n: \SH(k)^\veff(n) \to \SH(k)$ for the inclusion.
There is a right adjoint $\tilde{r}_n$ (see for example the proof of Lemma \ref{lemm:fn-trunc} below),
and we put $\tilde{f}_n = \tilde{i}_n
\tilde{r}_n$. This is the generalized slice filtration
\cite[Definition 5.5]{spitzweck2012motivic}. We denote by $\tilde{s}_n(E)$ a
cone on $\tilde{f}_{n+1} E \to \tilde{f}_n E$. This depends functorially on $E$:

\begin{lemm} \label{lemm:slices-functorial}
There exist a functor $\tilde{s}_0: \SH(k) \to \SH(k)$ and natural
transformations $p: \id \Rightarrow \tilde{s}_0$ and $\partial: \tilde{s}_0
\Rightarrow \tilde{f}_1[1]$, all determined up to unique
isomorphism, such that for each $E \in \SH(k)$ the following triangle is
distinguished: $\tilde{f}_1(E) \to \tilde{f}_0(E) \xrightarrow{p_E}
\tilde{s}_0(E) \xrightarrow{\partial_E} \tilde{f}_1E[1]$.

Moreover, for $E, F \in \SH(k)$ we have $[\tilde{f}_1(E)[1], \tilde{s}_0 F] =
0$.
\end{lemm}
\begin{proof}
By \cite[Proposition 1.1.9]{beilinson1982faisceaux} it
suffices to show the ``moreover'' part. We may as well show that if $E \in \SH(k)^\veff(1), F \in
\SH(k)$ then $[E[1], \tilde{s}_0 F] = 0$. Considering the long exact sequence
\[ [E[1], \tilde{f}_1 F] \xrightarrow{\alpha} [E[1], \tilde{f}_0 F] \to [E[1],
\tilde{s}_0 F] \to [E, \tilde{f}_1 F] \xrightarrow{\beta} [E, \tilde{f}_0 F] \]
it is enough to show that $\alpha$ and $\beta$ are isomorphisms. This is clear
since $E, E[1] \in \SH(k)^\veff(1)$ and $\tilde{f}_1 \tilde{f}_0 F \wequi
\tilde{f}_1 F$.
\end{proof}

The following lemmas will feature ubiquitously in the sequel. Recall that the
$f_i$ are triangulated functors.

\begin{lemm} \label{lemm:fn-T-smash}
For $E \in \SH(k)$ we have
\[ T \wedge f_n(E) \wequi f_{n+1}(T \wedge E) \]
and
\[ T \wedge \tilde{f}_n(E) \wequi \tilde{f}_{n+1}(T \wedge E). \]
\end{lemm}
\begin{proof}
We have $\SH(k)^\eff(n+1) = \SH(k)^\eff(n) \wedge T$. Now for $X \in
\SH(k)^\eff(n)$ we compute
\[ [X \wedge T, T \wedge f_n E] \iso [X, f_n E] \iso [X, E] \iso [X \wedge T, E \wedge T]
    \iso [X \wedge T, f_{n+1}(E \wedge T)]. \]
Thus $T \wedge f_n E \wequi f_{n+1} (T \wedge E)$ by the Yoneda lemma.

The proof
for $\tilde{f}_n$ is exactly the same, with $\SH(k)^\eff$ replaced by
$\SH(k)^\veff$ and $f_\bullet$ replaced by $\tilde{f}_\bullet$.
\end{proof}

We obtain a $t$-structure on $\SH(k)^\eff(n)$ by ``shifting'' the $t$-structure on
$\SH(k)^\eff$ by $\Gmp{n}$. In other words, if $E \in \SH(k)^\eff(n)$
then $E \in \SH(k)^\eff(n)_{\ge 0}$ if and only if $\ul{\pi}_i(E)_{-n} = 0$ for all
$i < 0$. Since $\wedge \Gmp{n}: \SH(k)^\eff \to \SH(k)^\eff(n)$ is an
equivalence of categories, all the properties established in the previous
section apply to $\SH(k)^\eff(n)$ as well, suitably reformulated. In particular
$\SH(k)^\eff(n)_{\ge 0}$ \emph{is} the non-negative part of a $t$-structure. We denote
the associated truncation by $E \mapsto E_{\ge_{e,n}0} \in \SH(k)^\eff(n)_{\ge
0}$, and so on.

\begin{lemm} \label{lemm:fn-shifted}
Denote by $j_n: \SH(k)^\eff(n+1) \to \SH(k)^\eff(n)$ the canonical inclusion.
The restricted functor $f_{n+1}: \SH(k)^\eff(n) \to \SH(k)^\eff(n+1)$ right adjoint
to $j_n$ and is $t$-exact, and $j_n$ is right $t$-exact.
\end{lemm}
\begin{proof}
Adjointness is clear. Let $E \in \SH(k)^\eff(n)$. We have $\ul{\pi}_i(f_{n+1}E)_{-n-1} =
\ul{\pi}_i(E)_{-n-1} = (\ul{\pi}_i(E)_{-n})_{-1}$. In particular if
$\ul{\pi}_i(E)_{-n} = 0$ then $\ul{\pi}_i(f_{n+1}E)_{-n-1} = 0$, which proves
that $f_{n+1}$ is $t$-exact. Then $j_n$ is right $t$-exact, being left adjoint
to a left $t$-exact functor.
\end{proof}

\begin{lemm} \label{lemm:fn-trunc}
Let $E \in \SH(k)$. Then
\[ \tilde{f}_nE \wequi i_n(r_n(E)_{\ge_{e,n} n}) \wequi f_n(E_{\ge n}). \]
\end{lemm}
\begin{proof}
Note that $\wedge T: \SH(k)^\eff(n) \to \SH(k)^\eff(n+1)$ induces equivalences
$\SH(k)^\eff(n)_{\ge n} \to \SH(k)^\eff(n+1)_{\ge n+1}$ and similarly for the
non-negative parts. It follows that for $E \in \SH(k)^\eff(n)$ we have
$E_{\ge_{e,n} n} \wedge T \wequi (E \wedge T)_{\ge_{e,n+1}n+1}$. Similarly we
find that
for $E \in \SH(k)$ we have $(E \wedge T)_{\ge n+1} \wequi E_{\ge n} \wedge T$.
Together with Lemma \ref{lemm:fn-T-smash} this implies that the current lemma
holds for some $n$ if and only if it holds for $n+1$ (and all $E$).
We may thus assume that $n=0$.

We have a factorisation
of inclusions $\SH(k)^\veff(0) \to \SH(k)^\eff(0) \to \SH(k)$ and hence the
right adjoint factors similarly. But the right adjoint to $\SH(k)^\veff(0) \to
\SH(k)^\eff(0)$ is truncation in the effective $t$-structure by definition, whence the
first equivalence. The second equivalence follows from $t$-exactness of $r$,
i.e. Proposition \ref{prop:effective-t-structure} part (3).
\end{proof}

From now on, we will write $f_n E_{\ge m}$ when convenient. This shall always
mean $f_n(E_{\ge m})$ and never $f_n(E)_{\ge m}$. This is the same
as $i_n(r_n(E)_{\ge_{e,n} m})$. In particular in calculations, we will similarly write $s_n E_{\ge m}$
to mean $s_n(E_{\ge m})$, never $s_n(E)_{\ge m}$. We will also from now on
mostly write $f_0$ in place of $i_0 r_0$; whenever we make statements like
``$f_0$ is $t$-exact'' we really mean that $f_0: \SH(k) \to \SH(k)^\eff$ is
$t$-exact.

\begin{lemm} \label{lemm:s0-decomp}
For $E \in \SH(k)$ there exist natural  triangles
\begin{equation} \label{eq:s0-decomp-1}
  s_0(E_{\ge 1}) \to \tilde{s}_0(E) \to f_0(\ul{\pi}_0(E)_*).
\end{equation}
and
\begin{equation} \label{eq:s0-decomp-2}
  f_1(\ul{\pi}_0(E)_*) \to \tilde{s}_0(E) \to s_0(E_{\ge 0}).
\end{equation}
\end{lemm}
Of course, there are variants of this lemma for $\tilde{s}_n$, obtained for
example by using $\tilde{s}_n(E) \wequi \tilde{s}_0(E \wedge T^{-\wedge n})
\wedge T^{\wedge n}$.
\begin{proof}
I claim that there are canonical isomorphisms (1) $s_0(E_{\ge 1}) \wequi
\tilde{s}_0(E)_{\ge_e 1}$, (2) $f_0(\ul{\pi}_0(E)_*) \wequi
\ul{\pi}_0^\eff(\tilde{s}_0 E)$, (3) $f_1(\ul{\pi}_0(E)_*) \wequi f_1
\tilde{s}_0(E)$ and (4) $s_0(E_{\ge 0}) \wequi s_0 \tilde{s}_0(E)$. Since
$\tilde{s}_0(E) \in \SH(k)^\eff_{\ge 0}$, the two purported triangles are thus
the functorial triangles $\tilde{s}_0(E)_{\ge_e 1} \to \tilde{s}_0(E) \wequi
\tilde{s}_0(E)_{\ge_e 0} \to \ul{\pi}_0^\eff \tilde{s}_0(E)$ and $f_1
\tilde{s}_0(E) \to \tilde{s}_0(E) \wequi f_0 \tilde{s}_0(E) \to s_0
\tilde{s}_0(E)$.

(1) We have $s_0(E_{\ge 1}) \wequi s_0 f_0(E_{\ge 1}) \wequi s_0
\tilde{f}_0(E)_{\ge_e 1}$, since $f_0$ is $t$-exact. In particular we may assume
that $E \in \SH(k)^\veff$. Consider the following commutative diagram
\begin{equation*}
\begin{CD}
                 @.               @.  \tilde{s}_0(E)_{\ge_e 1} \\
        @.              @.                @Vp'VV \\
\tilde{f}_1 E   @>>> E           @>>> \tilde{s}_0 E @>>> \tilde{f}_1(E)[1] @>>> E[1] \\
@|                 @ApAA                  @AhAA                 @| @A{p[1]}AA  \\
f_1(E_{\ge_e 1}) @>>> E_{\ge_e 1} @>>> s_0(E_{\ge_e 1}) @>>> f_1(E_{\ge_e 1} )[1] @>>> E_{\ge_e 1}[1]
\end{CD}
\end{equation*}
Here the rows are triangles and the maps $p$ and $p'$ are the
canonical ones. The map $h$ is induced. Using that $[f_1(E_{\ge_e 1} )[1],
\tilde{s}_0 E]$ by Lemma \ref{lemm:slices-functorial} and \cite[Proposition
1.1.9]{beilinson1982faisceaux}, we see that $h$ is in fact the \emph{unique} morphism
rendering the diagram commutative. We have $f_1(E_{\ge_e 1}) \in \SH(k)^\eff_{\ge
1}$, by Lemma \ref{lemm:fn-shifted}. It follows that $s_0(E_{\ge_e 1}) \in \SH(k)^\eff_{\ge 1}$. Of course
also $\tilde{s}_0(E)_{\ge_e 1} \in \SH(k)^\eff_{\ge 1}$. Let $T \in
\SH(k)^\eff_{\ge 1}$. Then $p'$ induces $[T, \tilde{s}_0(E)_{\ge_e 1}]
\xrightarrow{\iso} [T, s_0(E_{\ge_e 1})]$. Hence by Yoneda, it suffices to
show that $h$ induces $[T, s_0(E_{\ge_e 1})] \xrightarrow{\iso} [T,
\tilde{s}_0 E]$. By the 5-lemma it suffices to show that $[T, E_{\ge_e 1}]
\xrightarrow{\iso} [T, E]$ which is clear by definition, and that $[T,
E_{\ge_e 1}[1]] \to [T, E[1]]$ is injective. This follows from the exact
sequence $0= [T, E_{\le_e 0}] \to [T, E_{\ge_e 1}[1]] \to [T, E[1]]$.

(2) We have $\ul{\pi}_0^\eff(\tilde{s}_0 E) \wequi \ul{\pi}_0^\eff(\tilde{f}_0
E) \wequi \ul{\pi}_0^\eff(f_0 E)$,
since $\ul{\pi}_0^\eff(\tilde{f}_1E) = 0 = \ul{\pi}_{-1}^\eff(\tilde{f}_1(E))$,
again by Lemma \ref{lemm:fn-shifted}. We conclude since $f_0: \SH(k) \to
\SH(k)^\eff$ is $t$-exact by Proposition \ref{prop:effective-t-structure}(3).

(3) We have the two canonical triangles $\tilde{f}_0(E)_{\ge_e 1}
\to \tilde{f_0} E \to \ul{\pi}_0^\eff(f_0 E) \wequi f_0 \ul{\pi}_0(E)_*$ (using
$t$-exactness of $f_0$) and $\tilde{f}_1 E \to \tilde{f}_0 E \to \tilde{s}_0 E$.
The middle terms are canonically isomorphic, and the left terms become
canonically isomorphic after applying $f_1$. There is thus an induced
isomorphism $f_1 \ul{\pi}_0(E)_* \to f_1 \tilde{s}_0(E)$, which is in fact
unique by \cite[Proposition 1.1.9]{beilinson1982faisceaux}, provided we show
that $[\tilde{f}_1(E)[1], f_1 \tilde{s}_0(E)] = 0$. This follows from the exact
sequence $[\tilde{f}_1(E)[1], s_0(\tilde{s}_0(E))[1]] \to [\tilde{f}_1(E)[1],
f_1 \tilde{s}_0(E)] \to [\tilde{f}_1(E)[1], \tilde{s}_0(E)]$, Lemma
\ref{lemm:slices-functorial}, and the analogue of Lemma
\ref{lemm:slices-functorial} for ordinary slices. We can quickly prove this
analogue: if
$E \in \SH(k)^\eff(1)$ and $F \in \SH(k)$, then $[E, s_0(F)] = 0$, since $[E,
f_1 F[i]] \iso [E, f_0 F[i]]$ for all $i$.

(4) We have $s_0(E_{\ge 0}) \wequi s_0 f_0(E_{\ge 0}) \wequi s_0(f_0(E)_{\ge_e
0}) \wequi s_0\tilde{f}_0(E)$. Since $s_0$ is a triangulated functor, we have a
canonical triangle $s_0 \tilde{f}_1 E \to s_0 \tilde{f}_0 E \to
s_0 \tilde{s}_0 E$. Since $\tilde{f}_1E \in \SH(k)^\eff(1)$ we have $s_0
\tilde{f}_1 E  \wequi 0$, and hence we obtain the required isomorphism.
\end{proof}

\paragraph{Notation.} We define spectra for (generalized) motivic cohomology
theory as effective covers:
\begin{gather*}
  H_\mu \ZZ := f_0 \ul{K}^{M}_* \\
  H_\mu \ZZ/2 := f_0 \ul{K}^{M}_*/2 \\
  \tilde{H} \ZZ := f_0 \ul{K}^{MW}_* \\
  H_W \ZZ = f_0 \ul{K}^{W}_*.
\end{gather*}
Here $\ul{K}^{MW}_* \in \SH(k)^\heart$ denotes the homotopy module of Milnor-Witt
$K$-theory \cite[Chapter 3]{A1-alg-top}, i.e. $\ul{\pi}_0(S)_*$, where $S$ is
the sphere spectrum \cite[Theorem 6.40]{A1-alg-top}. Also $\ul{K}_*^W :=
\ul{K}_*^{MW}/h$ is the homotopy module of Witt $K$-theory \cite[Example 3.33]{A1-alg-top}.
Similarly
for the other right hand sides. This terminology is justified by the following
result.

\begin{lemm}
Let $S \to \ul{K}^M_*$ denote the composite $S \to \ul{\pi}_0(S)_* \simeq
\ul{K}_*^{MW} \to \ul{K}_*^M$. Then the canonical maps $s_0(S) \to
s_0(\ul{K}^M_*) \leftarrow f_0 \ul{K}^M_*$ are
equivalences. Moreover there is a canonical equivalence $f_0(\ul{K}^M_*)/2
\wequi f_0(\ul{K}^M_*/2)$, where on the left hand side we mean a cone on a
morphism in $\SH(k)$ and on the right hand side $\ul{K}^M_*/2 \in \SH(k)^\heart$
denotes the \emph{cokernel}.
\end{lemm}
In particular,
the spectra $H_\mu \ZZ$ and $H_\mu \ZZ/2$
represent motivic cohomology in the sense of Bloch's higher Chow groups
\cite[Theorems 6.5.1 and 9.0.3]{levine2008homotopy}. This agrees with
Voevodsky's definition of motivic cohomology \cite[Theorem
19.1]{lecture-notes-mot-cohom} (recall that our base field is perfect). Note
that if $C$ is a cone on $\ul{K}_*^M \xrightarrow{2} \ul{K}_*^M$, then $f_0(C)$
is a cone on $H_\mu \ZZ \xrightarrow{2} H_\mu \ZZ$, i.e. also canonically
isomorphic to $H_\mu \ZZ/2$. In other words in the notation $f_0 \ul{K}^M_*/2$
it does not matter if we view $\ul{K}^M_*/2$ as a cone or cokernel.
\begin{proof}
Since $\ul{K}^M_{-1} = 0$ we have $f_1 \ul{K}^M_* = 0$ and so $f_0 \ul{K}^M_*
\to s_0 \ul{K}^M_*$ is an equivalence. The spectrum $s_0 S$ represents motivic
cohomology \cite[Theorems 6.5.1 and 9.0.3]{levine2008homotopy} \cite[Theorem
19.1]{lecture-notes-mot-cohom} and hence $\ul{\pi}_0(s_0 S)_0 = \ZZ$ whereas
$\ul{\pi}_i(s_0 S)_0 = 0$ for $i \ne 0$. Similarly we have just from the
definitions that
$\ul{\pi}_0(f_0 \ul{K}^M_*)_0 = \ZZ$ and $\ul{\pi}_i(f_0 \ul{K}^M_*)_0 =
0$ for $i \ne 0$. Thus the map $s_0 S \to s_0 \ul{K}^M_* \wequi f_0 \ul{K}^M_*$ induces
an isomorphism on all $\ul{\pi}_i(\bullet)_0$, provided that $[S, s_0S] \to
[S, s_0 \ul{K}_*^M]$ is an epimorphism. But $[S, S] \to [S, s_0 S]$ and $[S, S]
\to [S, \ul{K}_*^M] = [S, f_0 \ul{K}_*^M] = [S, s_0 \ul{K}_*^M]$ are both
epimorphisms, so this is true. The first claim now follows from Proposition
\ref{prop:effective-t-structure}(1).

For the second claim, note that it follows from Propositions
\ref{prop:more-iheart}(3) and Lemma \ref{lemm:exactness} that $f_0:
\SH(k)^\heart \to \SH(k)^{\eff,\heart}$ is exact. Thus $f_0(\ul{K}^M_*/2)$ is
the cokernel of $\alpha: f_0(\ul{K}^M_*) \xrightarrow{2} f_0(\ul{K}^M_*) \in
\SH(k)^{\eff,\heart}$ and it remains to show that this cokernel is
isomorphic to the cone of $\alpha$. This happens if and only if $\alpha$ is
injective, which is clear since under the conservative exact functor from
Proposition \ref{prop:more-iheart}(3), $\alpha$ just corresponds to $\ZZ
\xrightarrow{2} \ZZ$.
\end{proof}

\paragraph{Philosophy.}
The complex realisation of $\Gm$ is $S^1$ and the complex realisation of $T$ is
$S^2$. We propose to think of the generalized slices as some kind of
``motivic (stable) 1-types''. Note that ordinary slices, as well as objects of
$\SH(k)^\heart$ and $\SH(k)^{\eff, \heart}$ would all be reasonable candidates
for ``motivic 0-types''. Triangle \eqref{eq:s0-decomp-1} shows that every
motivic 1-type can be canonically decomposed into an element of
$\SH(k)^{\eff,\heart}$ and a zero-slice (i.e. a birational motive) -- both of
which we think of as different kinds motivic 0-types. Thus in triangle
\eqref{eq:s0-decomp-1} we think of $f_0\ul{\pi}_0(E)_*$ as the $\pi_0$ part of the
motivic 1-type $\tilde{s}_0 E$ and of $s_0(E_{\ge 1})$ as the $\pi_1$ part of the
motivic 1-type. Of course, triangle \eqref{eq:s0-decomp-2} shows that every
motivic 1-type can be canonically decomposed into an element of $\SH(k)^{\eff,
\heart} \wedge \Gm$ and a zero-slice, which are again motivic 0-types. Thus in triangle
\eqref{eq:s0-decomp-2} we think of $s_0 E_{\ge 0}$ as the $\pi_0$-part of the
motivic 1-type and of $f_1 \ul{\pi}_0(E)_*$ as the $\pi_1$-part.

\section{The Generalized Slices of Hermitian K-Theory}
\label{sec:slices-herm-k}

We shall now compute $\tilde{s}_n(\KO)$. Recall that there exist motivic spaces
$GW^{[n]} \in Spc_*(k)$ which represent Hermitian $K$-theory and come
with a canonical weak equivalence
$\Omega_T GW^{[n]} \wequi GW^{[n-1]}$. Thus they can be assembled into a motivic
$T$-spectrum $\KO = (GW^{[0]}, GW^{[1]}, \dots) \in \SH(k)$ also representing Hermitian
$K$-theory \cite{hornbostel2005a1}. We remind that this spectrum is \emph{not} connective. We will write $\KO^{[n]} := \KO \wedge T^{\wedge n}$.
Observe that under the adjunction $\Sigma^\infty: Ho(Spc_*(k))
\leftrightarrows \SH(k): \Omega^\infty$ we have $\Omega^\infty(\KO^{[n]}) \wequi GW^{[n]}$.

By Bott periodicity, we have $GW^{[n+4]} \wequi GW^{[n]}$ and so $\KO^{[n+4]}
\wequi \KO^{[n]}$ \cite[Proposition 7]{schlichting2010mayer}.
Also recall the low-degree Hermitian $K$-groups
from Table \ref{tab:ko-groups}. This table can be deduced for example from the
identification of $GW^{[n]}_0$ with the Balmer-Walter Grothendieck-Witt groups
\cite[Lemma 8.2]{schlichting2016hermitian} and \cite[Theorem
10.1]{walter2003grothendieck}.

\begin{table}[tb]
\center
\begin{tabular}{c|c}
$n$ & $\ul{\pi}_0(GW^{[n]})$ \\
\hline
$0$ & $\ul{GW}$ \\
$1$ & $0$ \\
$2$ & $\ZZ$ \\
$3$ & $\ZZ/2$ \\
\end{tabular}
\caption{Low degree Hermitian $K$-groups as strictly homotopy invariant
Nisnevich sheaves.}
\label{tab:ko-groups}
\end{table}

\begin{lemm} \label{lemm:s0-KQH}
We have $f_0 \ul{\pi}_0(\KO^{[2]})_* \wequi H_\mu \ZZ$, and the natural induced map
$s_0 \KO^{[2]}_{\ge 0} = s_0 f_0 \KO^{[2]}_{\ge 0} \to s_0 f_0
\ul{\pi}_0(\KO^{[2]})_* \wequi s_0 H_\mu \ZZ \wequi H_\mu \ZZ$ is an isomorphism.
\end{lemm}
\begin{proof}
Consider the map $S \to \KO^{[2]}$ corresponding to $1 \in \ZZ = [S,
\KO^{[2]}]$. It induces $\alpha: \tilde{H}\ZZ = f_0 \ul{\pi}_0(S)_* \to f_0
\ul{\pi}_0(\KO^{[2]})_* \in \SH(k)^{\eff,\heart}$.
We also have the canonical map $\beta: \tilde{H}\ZZ \to
H_\mu \ZZ \in \SH(k)^{\eff,\heart}$. I claim that $\alpha$ and $\beta$ are
surjections with equal kernels. Indeed this may be checked after applying the
conservative exact functor from Proposition \ref{prop:more-iheart}(3), where
both maps correspond to the canonical map $\ul{GW} \to \ZZ$. It follows that
there is a canonical isomorphism $f_0 \ul{\pi}_0 \KO^{[2]} \wequi H_\mu \ZZ$.
(The point of this elaboration is that even
though we know that $f_0 \ul{\pi}_0 \KO^{[2]}$ and $H_\mu \ZZ$ are objects of
$\SH(k)^{\eff,\heart}$ which have the same underlying homotopy sheaf, a priori
the transfers could be different.)

The rest of the proof essentially uses an argument of Voevodsky \cite[Section
4]{voevodsky293possible}. Write
\[ \Sigma^\infty_s \SH^{S^1}(k) \leftrightarrows \SH(k): \Omega^\infty_s \]
for the canonical adjunction. Note that one may define a slice filtration for
$\SH^{S^1}(k)$ in just the same way as for $\SH(k)$: Let $\SH^{S^1}(k)(n) \subset \SH^{S^1}(k)$ denote the
localizing subcategory generated
by $T^n \wedge E$ for $E \in \SH^{S^1}(k)$. Then the inclusion $i'_n:
\SH^{S^1}(k)(n) \to \SH^{S^1}(k)$ has a right adjoint $r'_n$, one puts $f'_n =
i'_n r'_n$, and $s'_n E$ is defined to be the cofiber of $f'_{n+1} E \to f'_n E$
\cite[Section 7.1]{levine2008homotopy}.

It is enough to show that $f_0(\KO^{[2]}_{\ge 1})
\in \SH(k)^\eff(1)$, i.e. that $s_0(\KO^{[2]}_{\ge 1}) \wequi 0$. The functor
$\Omega^\infty_s: \SH(k)^\eff \to \SH^{S^1}(k)$ is conservative \cite[Lemma
3.3]{voevodsky293possible} and commutes with taking slices \cite[Theorems 9.0.3
and 7.1.1]{levine2008homotopy}. Since also $\Omega^\infty_s H_\mu \ZZ = \ZZ$ we
find that it is enough to show that $s_0^{S^1}(\Omega^\infty_s \KO^{[2]}_{\ge 0})
\wequi \ZZ$. (This is precisely how Voevodsky computes $s_0 KGL$, but since his
conjectures have been proved by Levine our result is unconditional.)

I claim that $\Omega^\infty \KO^{[2]}_{\ge 0} \wequi \Omega^\infty \KO^{[2]}
\wequi GW^{[2]} \in
Spc_*(k)$. To see this, note that for any $E \in \SH(k)$ and $U \in Spc_*(k)$ we have $[U,
\Omega^\infty E] = [\Sigma^\infty U, E] = [\Sigma^\infty U, E_{\ge 0}] = [U,
\Omega^\infty E_{\ge 0}]$, where for the middle equality we have used that
$\Sigma^\infty U \in \SH(k)_{\ge 0}$. Indeed as explained in Section
\ref{sec:recollections} this holds for $U = X_+$ with $X \in Sm(k)$, the
category $Spc_*(k)$ is generated under homotopy colimits by spaces of the form
$X_+$, the functor $\Sigma^\infty$ preserves homotopy colimits, and $\SH(k)_{\ge
0}$ is closed under homotopy colimits.

The geometric representability theorem of Panin-Walter
\cite{panin2010motivic} (for the case of symplectic $K$-theory, which is all we
need here) and Schlichting-Tripathi \cite{schlichting2015geometric} (for the
general case) implies that $GW^{[2]} \wequi \ZZ \times HGr$.
Thus by
\cite[Proposition 4.4 and proof of Lemma 4.6]{voevodsky293possible} the required
computation $s_0^{S^1}(\Omega^\infty \KO^{[2]}_{\ge 0})
\wequi \ZZ$ follows
from the next result (which is completely analogous
to \cite[Lemma 4.7]{voevodsky293possible}).
\end{proof}

\begin{lemm}
Let $HGr(m,n)$ denote the quaternionic Grassmannian of \cite{panin2010quaternionic},
with its canonical base point. Then $\Sigma^\infty_s HGr(m,n) \in
\SH^{S^1}(k)(1)$; in
other words there exists $E \in \SH^{S^1}(k)$ such that $\Sigma^\infty_s
HGr(m,n) \wequi E \wedge T$.
\end{lemm}
\begin{proof}
If $X$ is a smooth scheme, $U$ an open subscheme and $Z$ the closed complement
which also happens to be smooth, then by homotopy purity \cite[Theorem
3.2.23]{A1-homotopy-theory} there is a triangle
\[ \Sigma^\infty_s U_+ \to \Sigma^\infty_s X_+ \to \Sigma^\infty_s Th(N_{Z/X}), \]
where $N_{Z/X}$ denotes the normal bundle and $Th$ the Thom space (which is
canonically pointed). It follows from the octahedral axiom (for example) that we
may also use a base point inside $U \subset X$, i.e. that there is a  triangle
\[ \Sigma^\infty_s U \to \Sigma^\infty_s X \to \Sigma^\infty_s Th(N_{Z/X}) \]
(provided that $U$ is pointed, of course).

As a next step, if $E$ is a trivial vector bundle (of positive rank $r$) on $Z$ then
$\Sigma^\infty_s Th(E) \wequi T^{\wedge r} \wedge \Sigma^\infty Z_+ \in \SH^{S^1}(k) \wedge T$ is
1-effective. Since all vector bundles are Zariski-locally trivial and
$\SH^{S^1}(k) \wedge T$ is closed under homotopy colimits, the same holds
for an arbitrary vector bundle (of everywhere positive rank).
Consequently $Th(N_{Z/X}) \in \SH^{S^1}(k)
\wedge T$ and so $\Sigma^\infty_s X \in \SH^{S^1}(k) \wedge T$ if and only if
$\Sigma^\infty_s U \in \SH^{S^1}(k) \wedge T$.

We finally come to quaternionic Grassmannians. The space $HGr(m,n)$ has a closed
subscheme $N^+(m, n)$,
with open complement $Y(m,n)$ everywhere of positive codimension
\cite[Introduction]{panin2010quaternionic}. The space $N^+(m,n)$ is a vector bundle over
$HGr(m,n-1)$ \cite[Theorem 4.1(a)]{panin2010quaternionic} and so smooth.
The open complement $Y(m,n)$ is $\Aone$-weakly
equivalent to $HGr(m-1,n-1)$ \cite[Theorem 5.1]{panin2010quaternionic}.

Consequently $\Sigma^\infty_s HGr(m,n) \in \SH^{S^1}(k) \wedge T$ if and only if
$\Sigma^\infty_s HGr(m-1,n-1) \in \SH^{S^1}(k) \wedge T$. The claim is clear if
$m=0$, so the general case follows.
\end{proof}

We will use the following result, which is surely well known.

\begin{lemm} \label{lemm:triag-trunc}
Let $\mathcal{C}$ be a non-degenerate $t$-category and
\[ Z[-1] \xrightarrow{\partial} X \to Y \to Z \]
a  triangle.
If $\pi^\mathcal{C}_0(Y) \to \pi^\mathcal{C}_0(Z)$ is an epimorphism in
$\mathcal{C}^\heart$, then there is a unique map $\partial': Z_{\ge 0}[-1] \to
X_{\ge 0}$ such that the following diagram commutes, where all the unlabelled
maps are the canonical ones
\begin{equation*}
\begin{CD}
Z_{\ge 0}[-1] @>{\partial'}>> X_{\ge 0}@>>> Y_{\ge 0} @>>> Z_{\ge 0} \\
@VVV            @VVV @VVV          @VVV    \\
Z[-1]         @>{\partial}>> X @>>> Y         @>>> Z.
\end{CD}
\end{equation*}
Moreover, the top row is also a triangle.
\end{lemm}
\begin{proof}
Let $F$ be a homotopy fibre of $Y_{\ge 0} \to Z_{\ge 0}$. Since
$\pi_0^\mathcal{C} Y \to
\pi_0^\mathcal{C} Z$ is epi we have $\pi_{-1}^\mathcal{C} F = 0$ and $F \in \mathcal{C}_{\ge
0}$. We
shall show that $F \wequi X_{\ge 0}$.

By TR3, there is a commutative diagram
\begin{equation*}
\begin{CD}
Z_{\ge 0}[-1] @>>> F @>{\beta}>> Y_{\ge 0} @>>> Z_{\ge 0} \\
@VVV            @V{\alpha}VV @VVV          @VVV    \\
Z[-1]         @>>> X @>>> Y         @>>> Z.
\end{CD}
\end{equation*}
The composite $F \xrightarrow{\alpha} X \to X_{< 0}$ is zero because $F \in \mathcal{C}_{\ge 0}$
and consequently $\alpha$ factors as $F \xrightarrow{\alpha'} X_{\ge 0} \to X$.
By the five lemma, $\alpha'$ induces isomorphisms on all homotopy objects, so is
an isomorphism by non-degeneracy. This shows that in the above diagram, we may
replace $F$ by $X_{\ge 0}$ in such a way that $\alpha$ becomes the canonical
map. We need to show that then $\beta$ is also the canonical map. But since
$X_{\ge 0} \in \mathcal{C}_{\ge 0}$ we have $[X_{\ge 0}, Y_{\ge 0}] \iso [X_{\ge
0}, Y]$, and the image of $\beta$ in this latter group is the canonical map,
since the diagram commutes. This proves existence.

For uniqueness, note that the triangle
\[ X_{<0}[-1] \to X_{\ge 0} \to X \to X_{<0} \]
induces an exact sequence
\[ 0 = \Hom(Z_{\ge 0}[-1], X_{<0}[-1]) \to
    \Hom(Z_{\ge 0}[-1], X_{\ge 0}) \to \Hom(Z_{\ge 0}[-1], X), \]
whence there is indeed at most one map $\partial'$ making the square commute.
\end{proof}

We now come to the main result.

\begin{thm} \label{thm:main}
The generalized slices of Hermitian $K$-theory are given as follows:
\begin{equation*}
\tilde{s}_n \KO \wequi T^{\wedge n} \wedge
  \begin{cases}
  \tilde{s}_0(\KO) & n \equiv 0 \pmod{4} \\
  H_\mu \ZZ/2      & n \equiv 1 \pmod{4} \\
  H_\mu \ZZ        & n \equiv 2 \pmod{4} \\
  0                & n \equiv 3 \pmod{4}
  \end{cases}
\end{equation*}
Moreover the canonical decomposition \eqref{eq:s0-decomp-1} from Lemma
\ref{lemm:s0-decomp} of $\tilde{s}_0(\KO)$ is given by
\[ H_\mu \ZZ/2[1] \to \tilde{s}_0(\KO) \to \tilde{H}\ZZ, \]
and decomposition \eqref{eq:s0-decomp-2} is given by
\[ \Gm \wedge H_W \ZZ \to \tilde{s}_0 \KO \to H_\mu \ZZ. \]
\end{thm}
\begin{proof}
Since $\KO \wedge T^{\wedge 4} \wequi \KO$, the periodicity is clear, and we
need only deal with $n \in \{0,1,2,3\}$. It follows from Lemma \ref{lemm:fn-T-smash}
that $\tilde{s}_n \KO = T^{\wedge n} \wedge \tilde{s}_0(T^{\wedge -n} \wedge
\KO) = T^{\wedge n} \wedge \tilde{s}_0
\KO^{[-n]}$, and similarly for $s_n$.

We first deal with $n \in \{1,2,3\}$. Since then $\ul{\pi}_0(\KO^{[-n]})_{-1} =
0$ (see again Table \ref{tab:ko-groups}) we have $f_1 \ul{\pi}_0(\KO^{[-n]})_* =
0$ and hence by decomposition \eqref{eq:s0-decomp-2} from Lemma
\ref{lemm:s0-decomp} it is enough to show: $s_0(\KO^{[-1]}_{\ge 0}) = H_\mu
\ZZ/2, s_0(\KO^{[-2]}_{\ge 0}) = H_\mu
\ZZ$ and $s_0(\KO^{[-3]}_{\ge 0}) = 0$. The case $n=2$ is Lemma
\ref{lemm:s0-KQH}.

We shall now use the  triangle \cite[Theorem 4.4]{rondigs2013slices}
\begin{equation} \label{eq:wood-triangle}
  \Gm \wedge \KO \xrightarrow{\eta} \KO \xrightarrow{f} KGL
   \xrightarrow{h} \KO \wedge T = \KO^{[1]}.
\end{equation}

Smashing with $T^{\wedge 2}$ and applying $f_0$ we get (using that $T \wedge KGL \wequi KGL$)
\[ f_0 \KO^{[2]} \xrightarrow{f} f_0 KGL \xrightarrow{h} f_0 \KO^{[3]}. \]
I claim that $h: \ul{\pi}_0^\eff f_0KGL \to \ul{\pi}_0^\eff f_0\KO^{[3]} \in
\SH(k)^{\eff,\heart}$ is epi. By Proposition \ref{prop:more-iheart}(3), it
suffices to show that the induced map of Nisnevich sheaves
$h: \ZZ = \ul{\pi}_0(KGL)_0 \to \ul{\pi}_0(\KO^{[3]})_0 = \ZZ/2$ is an epimorphism.
All
symplectic forms have even rank, so the
map $f: \ul{\pi}_0(\KO^{[2]})_0 \to \ul{\pi}_0(KGL)_0 = \ZZ$ has image
$2\ZZ$ and thus non-zero cokernel. This implies the claim.
We may thus
apply Lemma \ref{lemm:triag-trunc} to $\mathcal{C} = \SH(k)^\eff$ and this
triangle. Consequently there is a triangle
\[ s_0 \KO^{[2]}_{\ge 0} \to s_0 KGL_{\ge 0} \to s_0 \KO^{[3]}_{\ge 0}. \]
Note that $f_0 KGL \in \SH(k)^{\eff}_{\ge 0}$ since $K$-theory of smooth schemes is
connective. Thus the triangle is isomorphic to
\[ H_\mu \ZZ \xrightarrow{2} H_\mu \ZZ \to s_0 \KO^{[3]}_{\ge 0} \]
yielding the required computation $s_0 \KO^{[3]}_{\ge 0} \wequi H_\mu \ZZ/2$.

Throughout the proof we will keep using Proposition \ref{prop:more-iheart}(3)
all the time. To simplify notation we will no longer talk about
$\ul{\pi}_i^\eff$ but only about $\ul{\pi}_i(\bullet)_0$; any statement about
the latter should be understood to correspond to a statement about the former.

By a similar argument, smashing with $T$ instead of $T^{\wedge 2}$, and
using that $\ZZ = \ul{\pi}_0(KGL)_0 \xrightarrow{h} \ul{\pi}_0(\KO^{[2]})_0 = \ZZ$
is an isomorphism, we conclude from 
\[ f_0 \KO^{[1]} \to f_0 KGL \to f_0 \KO^{[2]} \]
that $s_0 \KO^{[1]}_{\ge 0} = 0$.

We have thus handled the cases $n \in \{1,2,3\}$. Consider the triangle
\[ f_0 \KO^{[3]} \to f_0 KGL \to f_0 \KO^{[4]}, \]
obtained by smashing triangle \eqref{eq:wood-triangle} with $T^{\wedge 3}$ and
applying $f_0$.
We have $\ul{\pi}_0(\KO^{[3]})_0 = \ZZ/2$ and $\ul{\pi}_0(KGL)_0 = \ZZ$, whence
$\ul{\pi}_0(\KO^{[3]})_0 \to \ul{\pi}_0(KGL)_0$ must be the zero map and consequently
$\ul{\pi}_1(\KO^{[4]})_0 \to \ul{\pi}_0(\KO^{[3]})_0$ must be epi. It follows that we may
apply Lemma \ref{lemm:triag-trunc} to the rotated triangle
\[ f_0 KGL[-1] \to f_0\KO^{[4]}[-1] \to f_0 \KO^{[3]}. \]
Now $(E[-1])_{\ge 0} \wequi E_{\ge 1}[-1]$ and so, rotating back, we get a
triangle
\[s_0 \KO^{[3]}_{\ge 0} \to  s_0KGL_{\ge 1}  \to s_0\KO^{[4]}_{\ge 1}. \]
I claim that $s_0 KGL_{\ge 1} = 0$. Indeed we have a  triangle
\[ s_0 KGL_{\ge 1} \to s_0 KGL_{\ge 0} \to s_0 \ul{\pi}_0(KGL)_*, \]
and the two terms on the right are isomorphic by what we have already said.

We thus conclude that $s_0 \KO^{[0]}_{\ge 1} \wequi s_0 \KO^{[4]}_{\ge 1} \wequi s_0
\KO^{[3]}_{\ge 0}[1] \wequi H_\mu
\ZZ/2[1]$. The unit map $S \to \KO$ induces an isomorphism $\tilde{H}\ZZ = f_0
\ul{\pi}_0(S)_* \to f_0 \ul{\pi}_0(\KO)_*$, and hence the decomposition \eqref{eq:s0-decomp-1}
of $\tilde{s}_0 \KO$ follows.

It remains to establish the second decomposition. We first
show that $s_0(\KO_{\ge 0}) = H_\mu \ZZ$. For this we consider the
triangle
\[ f_0 \KO^{[0]} \to f_0 KGL \to f_0 \KO^{[1]} \]
obtained by applying $f_0$ to triangle \eqref{eq:wood-triangle}.
Since $\ul{\pi}_0(\KO^{[1]})_0 = 0$, by Lemma \ref{lemm:triag-trunc} we get a
triangle
\[ s_0 \KO^{[0]}_{\ge 0} \to s_0 KGL_{\ge 0} \to s_0 \KO^{[1]}_{\ge 0} \]
and we have already seen that $s_0 \KO^{[1]}_{\ge 0} = 0$ and $s_0 KGL_{\ge 0} =
H_\mu \ZZ$. Thus $s_0(\KO_{\ge 0}) = H_\mu \ZZ$ as claimed.

Finally the map $\ul{K}^W_* \wedge \Gm \xrightarrow{\eta} \ul{K}^{MW}_* \to
\ul{\pi}_0(\KO)_*$
(which makes sense since $\eta h = 0$ and so $\eta: \ul{K}_{*+1}^{MW} \to
\ul{K}_*^{MW}$ factors through $\ul{K}_{*+1}^{MW}/h = \ul{K}_{*+1}^W$; see also
\cite{morel2004puissances})
induces an isomorphism on
$\ul{\pi}_*(\bullet)_{-1}$ and consequently
\[ \Gm \wedge H_W \ZZ \wequi f_1(\ul{K}^W_* \wedge \Gm) \wequi f_1 \ul{K}^{MW}_*
\wequi f_1 \ul{\pi}_0(\KO)_*, \]
where the first equivalence is by Lemma \ref{lemm:fn-T-smash}. This concludes
the proof.
\end{proof}

\section{The Homotopy Sheaves of $\tilde{H}\ZZ$ and $H_W \ZZ$}
\label{sec:morel-conjecture}

In this section we prove the following result.

\begin{thm}[(Morel's Structure Conjecture\footnote{Morel conjectured a form of
this result in personal communication.})] \label{thm:morel-conjecture}
Let $k$ be a perfect field of characteristic different from two.

The natural maps $\tilde{H} \ZZ = f_0 \ul{K}^{MW}_* \to \ul{K}^{MW}_*$ and $H_W \ZZ
= f_0 \ul{K}^{W}_* \to \ul{K}^{W}_*$ induce isomorphisms on
$\ul{\pi}_0(\bullet)_*$. Moreover the natural maps $\tilde{H} \ZZ \to H_\mu \ZZ$
and $H_W \ZZ \to H_\mu \ZZ/2$ (obtained by applying $f_0$ to
$\ul{K}^{MW}_* \to \ul{K}^{MW}_*/\eta \wequi \ul{K}^M_*$ and
$\ul{K}^W_* \to \ul{K}^W_*/\eta \wequi \ul{K}^M_*/2$, respectively)
induce isomorphisms on $\ul{\pi}_i(\bullet)_*$ for $i \ne 0$.
\end{thm}

The proof will proceed through a series of lemmas. Note that $H_W\ZZ,
\tilde{H}\ZZ \in \SH(k)_{\ge 0}$ by Proposition
\ref{prop:effective-t-structure}(3), so in the theorem only $i \ge 0$ is
interesting.

Throughout this subsection, we fix the perfect field $k$ of characteristic not
two. Actually the only place where we explicitly use the assumption on the characteristic is
in Lemma \ref{lemm:HW-mod-eta} below.

\begin{lemm} \label{lemm:pi0-HW}
We have $\ul{\pi}_0(\tilde{H} \ZZ)_* = \ul{K}_*^{MW}$ and $\ul{\pi}_0(H_W \ZZ)_*
= \ul{K}^W_*$.
\end{lemm}
\begin{proof}
This follows from the results of the second half of Section
\ref{sec:effective-spectra}.
Namely the homotopy modules $\ul{K}^{MW}_*$ and
$\ul{K}^{W}_*$ are effective (Lemma \ref{lemm:htpy-mods-eff}), so
\[ \ul{\pi}_0(f_0\ul{K}^{MW}_*)_* = \ul{\pi}_0(ir\ul{K}^{MW}_*)_*
    \iso i^\heart r \ul{K}^{MW}_* \iso \ul{K}^{MW}_*, \]
(where the first equality is by definition, the second is by Proposition
\ref{prop:more-iheart} part (1), and the third is by part (2) of that
proposition and effectivity of $\ul{K}^{MW}_*$)
and similarly for $\ul{K}^W_*$.
\end{proof}

\begin{lemm} \label{lemm:HW-wt-one}
We have 
\begin{equation*}
\ul{\pi}_i(\Gm \wedge H_W \ZZ)_0 =
  \begin{cases}
    \ul{K}_1^{W} &i=0 \\
    \ZZ/2 &i=1 \\
    0 & \text{else}
  \end{cases}
\end{equation*}
The canonical map $H_W \ZZ \to H_\mu \ZZ/2$ induces an isomorphism on
$\ul{\pi}_1(\bullet)_1$.
\end{lemm}
\begin{proof}
By Theorem \ref{thm:main}, triangle \eqref{eq:s0-decomp-2} from Lemma \ref{lemm:s0-decomp} for $\tilde{s}_0
\KO$ reads
\[ \Gm \wedge H_W \ZZ \to \tilde{s}_0 \KO \to H_\mu \ZZ. \]
We know the $\ul{\pi}_i(\tilde{s}_0 \KO)_0$
from Theorem \ref{thm:main} (use triangle \eqref{eq:s0-decomp-1} from Lemma
\ref{lemm:s0-decomp} for $\tilde{s}_0 \KO$), and we know $\ul{\pi}_i(H_\mu \ZZ)_0$ from the
definition. We also know $\ul{\pi}_0(\Gm \wedge H_W \ZZ)_0$ from Lemma
\ref{lemm:pi0-HW}, and $\ul{\pi}_i(\Gm \wedge H_W \ZZ)_0 = 0$ for $i < 0$ by Proposition
\ref{prop:effective-t-structure}(3).
The claim
about the remaining $\ul{\pi}_i(H_W \ZZ)_0$ follows from the long exact sequence
of the triangle.

To prove the last claim, consider the diagram of homotopy modules
\begin{equation*}
\begin{CD}
\ul{K}^M_* @>h>> \ul{K}^{MW}_* @>>> \ul{K}^W_* \\
@|               @VVV           @VVV    \\
\ul{K}^M_* @>2>> \ul{K}^M_* @>>> \ul{K}^M_*/2.
\end{CD}
\end{equation*}
Applying $f_0$, both sequences become exact as sequences in
the abelian category $\SH(k)^{\eff,\heart}$. Equivalently, by Proposition
\ref{prop:more-iheart}(3), the sequences of sheaves
$\ZZ \to \ul{GW} \to \ul{W}$ and $\ZZ \to \ZZ \to \ZZ/2$ are exact. Thus we get
a morphism of triangles
\begin{equation*}
\begin{CD}
H_\mu \ZZ @>h>> \tilde{H}\ZZ @>>> H_W \ZZ \\
@|               @VVV           @VcVV    \\
H_\mu \ZZ @>2>> H_\mu \ZZ @>>> H\mu \ZZ/2,
\end{CD}
\end{equation*}
where $c$ is the canonical map we are interested in.
A diagram chase (using Lemma \ref{lemm:pi0-HW}) concludes that $\ul{\pi}_1(c)_1$
is an isomorphism as claimed.
\end{proof}

Recall that the motivic Hopf map $\eta: \mathbb{A}^2 \setminus 0 \to
\mathbb{P}^1$ induces a stable map of the same name $\eta: \Gm \to S$
\cite[Section 6.2]{morel-trieste}. For $E
\in \SH(k)$ we write $E/\eta$ for a cone on the map $\id \wedge \eta: E \wedge
\Gm \to E \wedge S \wequi E$. We denote by $\tau: H_\mu\ZZ/2[1] \to
H_\mu\ZZ/2 \wedge \Gm$ the unique non-zero map \cite[top of p. 33]{hoyois2013motivic}.

\begin{lemm} \label{lemm:HW-mod-eta}
There is an isomorphism $H_W \ZZ / \eta \wequi H_\mu \ZZ/2 \oplus H_\mu
\ZZ/2[2]$, such that $H_W \ZZ \to H_\mu \ZZ/2$ is the canonical map, and the boundary
map $H_\mu \ZZ/2[2] \to H_W \ZZ \wedge \Gm[1]$ has the property that the
composite
\[ H_\mu \ZZ/2[2] \to H_W \ZZ \wedge \Gm[1] \to H_\mu \ZZ/2 \wedge \Gm[1] \]
is $\tau[1]$.
\end{lemm}
\begin{proof}
It follows from the long exact sequence of homotopy sheaves and Lemma
\ref{lemm:HW-wt-one} that $\ul{\pi}_i(H_W \ZZ / \eta)_0$ is given by $\ZZ/2$ if $i
\in \{0, 2\}$ and is zero otherwise. Consider the triangle
\[(H_W \ZZ/\eta)_{\ge 1} \to H_W \ZZ/\eta \to  (H_W \ZZ/\eta)_{\le 0}. \]
We have $H_W \ZZ/\eta = s_0(H_W \ZZ)$ (see also
Proposition \ref{prop:slices-HW}(2)) and so $H_W \ZZ/\eta$ is an \emph{effective
motive}. We may consider the above triangle as coming from the homotopy
$t$-structure on $\DM^\eff$.
Since the heart of the category of effective motives
can be modeled as homotopy invariant sheaves with transfers (here we need the
assumption on the characteristic), we
find that $(H_W \ZZ/\eta)_{\ge 1} \wequi H_\mu \ZZ/2[2]$ and $(H_W
\ZZ/\eta)_{\le 0} \wequi H_\mu \ZZ/2$.
Since $\Hom_{\DM^\eff}(H_\mu \ZZ/2, H_\mu \ZZ/2[3]) =
0$ the triangle splits.

The composite $\alpha: H_\mu \ZZ/2[2] \to H_W \ZZ \wedge \Gm[1] \to H_\mu \ZZ/2 \wedge
\Gm[1]$ defines a cohomology operation of weight $(0,1)$. By the computation
of the motivic Steenrod algebra \cite[Theorem 1.1 (1)]{hoyois2013motivic},
this is either $\tau[1]$ or $0$. Here again we use the assumption that
$char(k) \ne 2$.
Consider again the triangle $H_W \ZZ \wedge \Gm
\to H_W \ZZ \to H_W \ZZ/\eta \wequi H_\mu \ZZ/2 \oplus H_\mu \ZZ/2[2]$, and its
morphism to the triangle for $H_\mu\ZZ/2 /\eta$. Since $\ul{\pi}_1(H_W \ZZ)_0 =
0$ and $\ul{\pi}_1(H_W \ZZ \wedge \Gm) = \ZZ/2$, the boundary map $\ZZ/2 =
\ul{\pi}_2(\ZZ/2[2])_0 \to \ul{\pi}_1(H_W \ZZ \wedge \Gm)_0$ must be an isomorphism.
Since also $\ul{\pi}_1(H_W \ZZ \wedge \Gm)_0 \to \ul{\pi}_1(H_\mu \ZZ/2 \wedge
\Gm)_0$ is an isomorphism by the last sentence of Lemma \ref{lemm:HW-wt-one}, we
conclude that $\alpha$ is not the zero map. This was to be shown.
\end{proof}

We will use the following easy fact about ``split triangles''.
\begin{lemm}\label{lemm:triang-split-trick}
Let $\mathcal{C}$ be a triangulated category and
\[ A \to B \oplus C \to D \]
a  triangle. Then $A \to B$ is an isomorphism if and only if $C \to D$ is an isomorphism.

In particular if $0 \to A \to B \oplus C \to D \to 0$ is an exact sequence in an
abelian category, then $A \to B$ is an isomorphism if and only if $C \to D$ is
an isomorphism.
\end{lemm}
\begin{proof}
Since the axioms of triangulated categories are self dual, it suffices to show
one implication. Thus suppose that $C \to D$ is an isomorphism.

Let $E \in \mathcal{C}$. Consider the long exact sequence
\[ \dots \to [E, D[-1]] \xrightarrow{\partial}
  [E, A] \to [E, B] \oplus [E, C] \to [E, D]
  \xrightarrow{\partial} [E, A[1]] \to \dots. \]
Since $C \to D$ is an isomorphism $[E, C[i]] \to [E, D[i]]$ is surjective, and
thus $\partial \equiv 0$.
We hence get a short exact sequence
\[ 0 \to [E, A] \xrightarrow{(i, j)} [E, B] \oplus [E, C] \xrightarrow{(p\,q)} [E, D] \to 0. \]
But in such a situation surjectivity of $q$ implies surjectivity of $i$, and
injectivity of $q$ implies injectivity of $i$. Hence $[E, A] \to [E, B]$ is an
isomorphism. Since $E$ was arbitrary, we conclude by the Yoneda lemma.
\end{proof}

\begin{lemm} \label{lemm:HW-pi-ge1}
The canonical map $H_W \ZZ \to H_\mu \ZZ/2$ induces an isomorphism on
$\ul{\pi}_k(\bullet)_n$ for $k \ge 1$ and $n \in \ZZ$.
\end{lemm}
\begin{proof}
Consider the long exact sequence of homotopy sheaves associated with the
triangle $H_W \ZZ \wedge \Gm \to H_W \ZZ \to H_\mu \ZZ/2 \oplus H_\mu \ZZ/2[2]$:
\begin{gather*}
\dots \to \ul{\pi}_{k+1}(H_W \ZZ)_n
 \to \ul{\pi}_{k+1} (H_\mu \ZZ/2)_n \oplus \ul{\pi}_{k-1}(H_\mu \ZZ/2)_n
 \to \ul{\pi}_k(H_W \ZZ)_{n+1} \\
 \to \ul{\pi}_k(H_W \ZZ)_n
 \to \ul{\pi}_k(H_\mu \ZZ/2)_n \oplus \ul{\pi}_{k-2}(H_\mu \ZZ/2)_n
 \to \dots
\end{gather*}
We shall use induction on $n$. If $n=0$ (or $n \le 0$) the claim is clear. We may assume by
induction that $\ul{\pi}_{k}(H_W \ZZ)_n \to \ul{\pi}_{k} (H_\mu \ZZ/2)_n$ is
an isomorphism for $k \ge 1$. It follows that (still for $k \ge 1$) we get short
exact sequences
\[ 0 \to \ul{\pi}_{k+1}(H_W \ZZ)_n
   \to \ul{\pi}_{k+1} (H_\mu \ZZ/2)_n \oplus \ul{\pi}_{k-1}(H_\mu \ZZ/2)_n
   \to \ul{\pi}_k(H_W \ZZ)_{n+1} \to 0. \]
Lemma \ref{lemm:triang-split-trick} now implies that $\ul{\pi}_{k-1}(H_\mu
\ZZ/2)_n \to \ul{\pi}_k(H_W \ZZ)_{n+1}$ is an isomorphism. By the last
part of Lemma \ref{lemm:HW-mod-eta} the composite
\[ \ul{\pi}_{k-1}(H_\mu \ZZ/2)_n \to \ul{\pi}_k(H_W \ZZ)_{n+1}
   \to \ul{\pi}_k(H_\mu \ZZ/2)_{n+1} \]
(where the last map is the canonical one) is $\tau$ and hence is an isomorphism
for $k \ge 1$ by Voevodsky's resolution of the Milnor conjectures \cite{voevodsky2003motivic}
(which implies that $H^{*,*}(k, \ZZ/2) = K_*^M/2[\tau]$).

It follows that $\ul{\pi}_k(H_W \ZZ)_{n+1} \to \ul{\pi}_k(H_\mu \ZZ/2)_{n+1}$
must also be an isomorphism. This concludes the induction and hence the proof.
\end{proof}

\begin{proof}[of Theorem \ref{thm:morel-conjecture}]
It remains to prove the claim about $\ul{\pi}_i(\tilde{H}\ZZ)_*$ for $i > 0$.
The exact sequence $0 \to \ul{K}_*^{MW} \to \ul{K}_*^M \oplus \ul{K}_*^W \to
\ul{K}_*^M/2 \to 0$ \cite[Théorème 5.3]{morel2004puissances} induces a  triangle
\[ \tilde{H} \ZZ \to H_W \ZZ \oplus H_\mu \ZZ \to H_\mu \ZZ/2. \]
By Lemma \ref{lemm:pi0-HW}, the map $\ul{\pi}_0(\tilde{H} \ZZ)_* \to \ul{\pi}_0 ( H_W \ZZ \oplus H_\mu
\ZZ)_*$ is just the canonical map $\ul{K}_*^{MW} \to \ul{K}_*^M \oplus
\ul{K}_*^W$ and thus injective. Hence by Lemma \ref{lemm:triag-trunc} we get a
 triangle
\[ (\tilde{H} \ZZ)_{\ge 1} \to (H_W \ZZ)_{\ge 1} \oplus (H_\mu \ZZ)_{\ge 1} \to
   (H_\mu \ZZ/2)_{\ge 1}. \]
Since $(H_W \ZZ)_{\ge 1} \to (H_\mu \ZZ/2)_{\ge 1}$ is an equivalence by Lemma
\ref{lemm:HW-pi-ge1}, we find that $(\tilde{H} \ZZ)_{\ge 1} \to (H_\mu \ZZ)_{\ge
1}$ is an equivalence by Lemma \ref{lemm:triang-split-trick}. This concludes the
proof.
\end{proof}

\paragraph{The Slices of $\tilde{H}\ZZ$ and $H_W \ZZ$.}
A minor extension of part the argument of the above proof also yields the
following.

\begin{prop} \label{prop:slices-HW}
\begin{enumerate}[(1)]
\item The canonical map $\tilde{H} \ZZ \to H_W \ZZ$ induces an isomorphism on $f_n$
and $s_n$ for $n \ge 1$.
\item For $n \ge 0$ we have $f_n H_W \ZZ \wequi
\Gmp{n} H_W \ZZ$.
\item We have $s_0 \tilde{H}\ZZ \wequi H_\mu \ZZ \oplus
H_\mu \ZZ/2[2]$, and for $n \ge 0$ we have $s_n H_W \ZZ \wequi \Gmp{n}
(H_\mu \ZZ/2 \oplus H_\mu \ZZ/2[2])$.
\end{enumerate}
\end{prop}
\begin{proof}
Since $\ul{\pi}_i(\tilde{H} \ZZ)_{-n} \to \ul{\pi}_i(H_W \ZZ)_{-n}$ is an
isomorphism for $n \ge 1$, claim (1) is clear. For claim (2), it follows from
Lemma \ref{lemm:fn-T-smash} that $f_n H_W \ZZ \wequi \Gmp{n} f_0(H_W \ZZ
\wedge \Gmp{-n})$.
But $\eta^n: H_W \ZZ \to H_W \ZZ \wedge \Gm{-n}$ induces an isomorphism
on $\ul{\pi}_i(\bullet)_0$ for $n \ge 0$, so $f_0(H_W \ZZ
\wedge \Gmp{-n}) \wequi f_0(H_W \ZZ) = H_W \ZZ$.

To prove (3), using the  triangle $\tilde{H} \ZZ \to H_W \ZZ
\oplus H_\mu \ZZ \to H_\mu \ZZ/2$ one finds using Lemma \ref{lemm:triang-split-trick} that the
claim for $s_0 \tilde{H} \ZZ$ reduces to the one for $H_W \ZZ$. By (2), $s_n
H_W \ZZ \wequi (\Gmp{n+1} H_W \ZZ)/\eta \wequi \Gmp{n+1}(H_W
\ZZ/\eta)$, and $H_W \ZZ / \eta \wequi H_\mu \ZZ/2 \oplus H_\mu \ZZ/2[2]$ by
Lemma \ref{lemm:HW-mod-eta}.
\end{proof}

\bibliographystyle{plainc}
\bibliography{bibliography}

\end{document}